\documentclass{article}
\usepackage{fullpage}
\usepackage{proof,prooftree}
\usepackage{amsfonts,amssymb}
\usepackage{latexsym}
\usepackage{pslatex}
\usepackage{stmaryrd}
\usepackage{fancybox}
\usepackage{amsmath} %
\usepackage{amssymb,prooftree,proof}
\usepackage{latexsym}
\usepackage[all]{xy}
\usepackage{qsymbols}
\usepackage{pslatex}
\usepackage[dvips]{graphicx}
\usepackage{url}
\usepackage{xspace}
\usepackage{authblk}

\newenvironment{keyword}[1]{\medskip\noindent\textbf{Keywords:}\ #1}

\newcount\timehh\newcount\timemm \timehh=\time
\divide\timehh by 60 \timemm=\time
\newcommand{\timestamp}{
{\protect\small\sl\today\ --
  \ifnum\timehh<10 0\fi\number\timehh\,:\,
  \ifnum\timemm<10 0\fi\number\timemm}}

\newif\ifcomment

\newcommand{\via}[1]{}
\newcommand{\SKIP}[1]{}

\newcommand{\mmu}{\widetilde{\mu}}
\newcommand{\lm}{\ensuremath{\lambda \mu}}
\newcommand{\lmm}{\ensuremath{\overline{\lambda} \mu\mmu}}

\newcommand{\lG}{\ensuremath{\lambda^{\mathsf{Gtz}}}}

\newcommand{\llG}{\lambda_{\circledR}^{\mathsf{Gtz}}}

\newcommand{\rcl}{\lambda_{\circledR}}

\newcommand{\llxr}{\lambda \mathsf{lxr}}

\newcommand{\isub}[2]{[#1/#2]}   
\newcommand{\app}[2]{#1@#2}

\newcommand{\bindx}{\widehat x.}
\newcommand{\bindy}{\widehat y.}
\newcommand{\bindz}{\widehat z.}

\newcommand{\append}[2]{#1@#2}

\newcommand{\intt}[1]{\lfloor #1 \rfloor^{\mathcal{R}}}
\newcommand{\intc}[1]{\lfloor #1 \rfloor_k^{\mathcal{R}}}

\newcommand{\weak}[2]{#1 \odot #2}
\newcommand{\cont}[4]{#1 < ^{#2}_{#3}#4}

\newcommand{\size}[1]{\mathcal{S}^{\mathcal{R}}(#1)}
\newcommand{\cnorm}[1]{|\!|#1|\!|_\mathsf{c}^{\mathcal{R}}}
\newcommand{\wnorm}[1]{|\!|#1|\!|_\mathsf{w}^{\mathcal{R}}}

\newcommand{\isubs}[2]{#1/#2}   
\newcommand{\tlam}{{Types}} 
\newcommand{\lefti}{[ \! [}      
\newcommand{\righti}{] \! ]}
\newcommand{\ti}[1]{\lefti #1 \righti^{\mathcal{R}}}
\newcommand{\tei}{\ti} 
\newcommand{\SN}{\ensuremath{\mathcal{SN}}} 
\newcommand{\SNc}{\ensuremath{\mathcal{SN}}_{\mathcal{R}}^{\circ}} 
\newcommand{\SNR}{\ensuremath{\mathcal{SN}}_{\mathcal{R}}} 
\newcommand{\INH}{\textsf{INH}}
\newcommand{\VAR}{\textsf{VAR}}
\newcommand{\SAT}{\textsf{SAT}}

\newcommand{\vX}{\mathcal{X}} 
\newcommand{\vM}{\mathcal{M}} 
\newcommand{\vN}{\mathcal{N}} 
\newcommand{\tA}{\alpha} 
\newcommand{\tB}{\beta} 
\newcommand{\tC}{\gamma} 
\newcommand{\tD}{\delta} 
\newcommand{\tE}{\epsilon} 
\newcommand{\tS}{\sigma} 
\newcommand{\tT}{\tau} 
\newcommand{\tR}{\rho} 
\newcommand{\tU}{\upsilon}
\newcommand{\ggr}{\gg^{\mathcal{R}}}

\newcommand{\fsto}{\xymatrix@C=15pt{\ar@{>}[r] &}}

\newcommand{\f}{f}

\newcommand{\lGR}{\lambda^{\mathsf{Gtz}}_{\mathcal{R}}}
\newcommand{\lGo}{\lambda^{\mathsf{Gtz}}_{\emptyset}}
\newcommand{\lGc}{\lambda^{\mathsf{Gtz}}_{\mathsf{c}}}
\newcommand{\lGw}{\lambda^{\mathsf{Gtz}}_{\mathsf{w}}}
\newcommand{\lGcw}{\lambda^{\mathsf{Gtz}}_{\mathsf{c}\mathsf{w}}}
\newcommand{\co}{\mathsf{c}}
\newcommand{\w}{\mathsf{w}}
\newcommand{\TGR}{\mathrm{T}^{\mathsf{Gtz}}_{\mathcal{R}}} 
\newcommand{\KGR}{\mathrm{K}^{\mathsf{Gtz}}_{\mathcal{R}}} 
\newcommand{\LGR}{\Lambda^{\mathsf{Gtz}}_{\mathcal{R}}} 
\newcommand{\LGo}{\Lambda^{\mathsf{Gtz}}_{\emptyset}}
\newcommand{\LGc}{\Lambda^{\mathsf{Gtz}}_{\mathsf{c}}}
\newcommand{\LGw}{\Lambda^{\mathsf{Gtz}}_{\mathsf{w}}}
\newcommand{\LGcw}{\Lambda^{\mathsf{Gtz}}_{\mathsf{c}\mathsf{w}}}
\newcommand{\unc}{\sqcup_{\mathsf{c}}}

\newcommand{\vdashr}{\vdash_{\mathcal{R}}}
\newcommand{\modelsr}{\models_{\mathcal{R}}}

\newcommand{\lR}{\lambda_{\mathcal{R}}}
\newcommand{\lo}{\lambda_{\emptyset}}
\newcommand{\lc}{\lambda_{\mathsf{c}}}
\newcommand{\lw}{\lambda_{\mathsf{w}}}
\newcommand{\lcw}{\lambda_{\mathsf{c}\mathsf{w}}}
\newcommand{\LR}{\Lambda_{\mathcal{R}}}
\newcommand{\LRc}{\Lambda_{\mathcal{R}}^{\ensuremath{\circ}}}
\newcommand{\LRvdash}[2]{\ensuremath{\LR{\scriptscriptstyle(#1\vdash_{\lR} #2)}}}

\usepackage[usenames,dvipsnames]{color}     

\newif\ifcomment

\commentfalse           

\newcommand{\comment}[2]{\ifcomment
                         \begin{bf}#1: \color{blue}{#2}
                           ***\end{bf}
                         \fi}

\newcommand{\sil}[1]{\ifcomment {\color{Green} {#1}} \else #1\fi}

\newcommand{\silv}[1]{\ifcomment {\color{RoyalPurple} {#1}} \else #1\fi}

\definecolor{darkgreen}{rgb}{0,.5,0}
\newcommand{\jel}[1]{\ifcomment {\color{darkgreen} {#1}} \else #1\fi}
\definecolor{darkbrown}{cmyk}{.3,.75,.75,.15}
\newcommand{\pierre}[1]{\ifcomment {\color{darkbrown} {#1}} \else
  #1\fi}
\newcommand{\compier}[1]{\comment{\color{RedOrange}{Pierre}}{#1}}

\newcommand{\ead}[1]{\thanks{\textsf{Email:}~#1}}

\begin{document}
\newtheorem{exa}{Example}
\newtheorem{prop}{Proposition}
\newtheorem{lem}{Lemma}
\newtheorem{thm}{Theorem}
\newtheorem{rem}{Remark}
\newtheorem{defi}{Definition}
\newenvironment{proof}[1]{\begin{quotation}\noindent\textsf{Proof:} #1}%
{\(\Box\)\end{quotation}}

\title{Resource control and strong normalisation}

\author[1]{S. Ghilezan \ead{gsilvia@uns.ac.rs}}
\author[1]{J. Iveti\' c  \ead{jelenaivetic@uns.ac.rs}}
\author[2]{P. Lescanne \ead{pierre.lescanne@ens-lyon.fr}}
\author[3]{S. Likavec \ead{likavec@di.unito.it}}
\affil[1]{University of Novi Sad, Faculty of Technical Sciences,  Serbia}
\affil[2]{University of Lyon, \' Ecole Normal Sup\' erieure de Lyon, France}
\affil[3]{Dipartimento di Informatica, Universit\`a di Torino, Italy}
%


\date{\today}
\maketitle

\ifcomment \centerline{\timestamp} \fi

\begin{abstract}
We introduce the \emph{resource control cube}, a system
consisting of eight intuitionistic lambda calculi with either
implicit or explicit control of resources and with either natural
deduction or sequent calculus. The four calculi of the
cube that correspond to natural deduction have been proposed by
Kesner and Renaud and the four calculi that correspond to sequent lambda
calculi are introduced in this paper.
The presentation is parameterized  with the set of resources (weakening or
contraction), which enables a uniform treatment of the eight
calculi of the cube.
The simply typed resource control cube, on the one hand,
expands the Curry-Howard correspondence to intuitionistic natural
deduction and intuitionistic sequent logic with implicit or
explicit structural rules and, on the other hand, is related to
substructural logics.

We propose a general intersection type system for the resource control cube calculi.
Our main contribution is a characterisation of strong normalisation of reductions in this cube.
First, we prove that typeability implies strong normalisation in the ``natural deduction base"
of the cube by adapting the reducibility method. We then prove that typeability implies strong
normalisation in the ``sequent base" of the cube by using a combination of well-orders and a
suitable embedding in the ``natural deduction base". Finally, we prove that strong normalisation
implies typeability in the cube using head subject expansion. All proofs are general and can be
made specific to each calculus of the cube by instantiating the set of resources.


\medskip

\textit{1998 ACM Subject Classification: } F.4.1 [Mathematical Logic]: Lambda calculus and related systems, F.3.1
  [Specifying and Verifying and Reasoning about Programs]: Logics of programs, F.3.2
  [Semantics of Programming Languages]: Operational semantics, F.3.3 [Studies of
  Program Constructs]: Type structure

\end{abstract}

\section*{Introduction}
\label{sec:intro}

Curry--Howard correspondence or formulae-as-types and proofs-as-programs paradigm~\cite{howa80}, establishes a fundamental connection between various logical and computational systems.
Simply typed $\lambda$-calculus provides the computational interpretation of
intuitionistic natural deduction where simplifying a proof corresponds to a program execution.
The correspondence between sequent calculus derivations and natural deduction derivations
is not a one-to-one map: several cut-free derivations correspond to one normal derivation~\cite{bareghil00}.
Starting from Griffin's extension of the Curry--Howard correspondence to classical logic~\cite{grif90},
this connection has been extended to other calculi and logical systems.
For instance, Parigot's \lm-calculus~\cite{pari92} corresponds to classical natural deduction and as such inspired investigation into the relationship between classical logic and theories of control
in programming languages~\cite{pari97,groo94,ongstew97,bier98,arioherb03,herbghil08}.
In the realm of sequent calculus,
Herbelin's $\overline{\lambda}$-calculus~\cite{herb95b} and
Esp\'{\i}rito Santo's $\lG$-calculus~\cite{jesTLCA07} correspond to intuitionistic sequent calculus,
whereas Barbanera and Berardi's symmetric calculus~\cite{barbbera96} and
Curien and Herbelin's $\lmm$-calculus~\cite{curiherb00} correspond to its classical variants.
An extensive overview of this subject can be found in~\cite{soreurzy06,ghillika08}.

Our intention is to bring this correspondence to the calculi with control operators,
namely erasure and duplication, which correspond to weakening and contraction on the logical side.
The wish to control the use of variables in a $\lambda$-term can be traced back to
Church~\cite{ChurchCalculiOfLambdaConversion} who introduced
the $\lambda I$-calculus. According to Church, the variables bound
by $\lambda$ abstraction should occur in the term at least once.
More recently, following the ideas of linear logic~\cite{LL},
van Oostrom~\cite{oost01} proposed to extend the $\lambda$-calculus,
and Kesner and Lengrand~\cite{kesnleng07}
proposed to extend the $\lambda$x-calculus with explicit
substitution, with operators to tightly control the use of
variables (resources). Resource control in sequent calculus corresponding to classical logic was proposed in \cite{zunicPHD}. Resource control
both in $\lambda$-calculus and  $\lambda$x-calculus is proposed in
\cite{kesnrena09,kesnrena11}, whereas resource control for sequent $\lambda$-calculus is proposed in
\cite{ghilivetlesczuni11,ghilivetlesclika11}. Like in the $\lambda
I$-calculus, bound variables must still occur in the term, but if
a variable $x$ is not used in a term $M$ this can be expressed by
using the expression $\weak{x}{M}$ where the operator $\odot$ is
called \emph{erasure (aka weakening)}. In this way, the term $M$
does not contain the variable $x$, but the term $\weak{x}{M}$
does. Similarly, a variable should not occur twice. If
nevertheless, we want to have two positions for the same variable,
we have to duplicate it explicitly, using fresh names. This is
done by using the operator $\cont{x}{x_1}{x_2}{\ }$, called
\emph{duplication (aka contraction)} which creates two fresh
variables $x_1$ and $x_2$. Extending the classical
$\lambda$-calculus and the sequent $\lambda$-calculus $\lG$ with
explicit erasure and duplication provides the Curry--Howard
correspondence for intuitionistic natural deduction and
intuitionistic sequent calculus with explicit structural rules, as
investigated in~\cite{kesnleng07,kesnrena09,kesnrena11,ghilivetlesczuni11}.
The notation used in this paper is along the lines of \cite{zunicPHD} and close to \cite{oost01}.

A different approach to the resource aware lambda calculus,
motivated mostly by the development of the process calculi,
was investigated by Boudol in~\cite{boud93}. Instead of extending the
syntax of $\lambda$-calculus with explicit resource operators, Boudol
proposed a non-deterministic calculus with a generalised notion of application.
In his work, a function is applied to a structure called a bag, having the form
$(N^{m_1}_1|...|N^{m_k}_k)$ in which $N_i,\;i=1,...,k$ are resources and
$m_i \in \mathbb{N}\cup \{\infty\},\;i=1,...,k$ are multiplicities, representing
the maximum possible number of the resource usage. In this framework, the usual
application is written as $MN^{\infty}$. The theory was further developed
in~\cite{boudcurilava99}, connected to linear logic via differential
$\lambda$-calculus in~\cite{ehrhregn03} and even typed with non-idempotent intersection types in~\cite{pagaronc10}.

Our contribution is inspired by Kesner and Lengrand's~\cite{kesnleng07} and Kesner and Renaud's~\cite{kesnrena09,kesnrena11} work on resource operators for
$\lambda$-calculus. The linear $\llxr$ calculus of Kesner and Lengrand introduces operators for substitution, erasure and duplication, preserving at the same time strong
normalisation, confluence and subject reduction property of its predecessor ${\lambda
\mathsf{x}}$~\cite{BlooRose95}. This approach is then generalised in Kesner and Renaud's \emph{Prismoid of Resources}~\cite{kesnrena09,kesnrena11} by considering various combinations in which  these three operators are made explicit or implicit.
In this paper we introduce the notion of \emph{resource control cube}, consisting of two systems $\lR$ and $\lGR$, which can be seen as two opposite sides of the cube (see Figure~\ref{fig:cube}).
In $\lR$ we consider four calculi, each of which is obtained by making erasure (aka
weakening) or duplication (aka contraction) explicit or implicit.
In $\lGR$ we consider the four sequent style counterparts of the $\lR$-calculi.
We will call the $\lR$ system \emph{the natural deduction ND-base of the cube}, and the $\lGR$ system \emph{the sequent LJ-base of the cube}\footnote{We are aware that our notation is not symmetric, since $\mathsf{Gtz}$ is specified explicitly for the LJ-base but $\mathsf{ND}$ is not for the ND-base. This was a conscious decision because ND-base contains ordinary $\lambda$-calculus where correspondence with natural deduction is not explicitly written.}.
Explicit control of erasure and
duplication
leads to decomposing of reduction steps into more atomic ones, thus revealing details of
computation which are usually left implicit.  Since erasing and duplicating of (sub)terms
essentially changes the structure of a program, it is important to see how this mechanism
really works and to be able to control this part of computation.

\begin{figure}[hbtp]
\begin{center}
\includegraphics[scale=0.7]{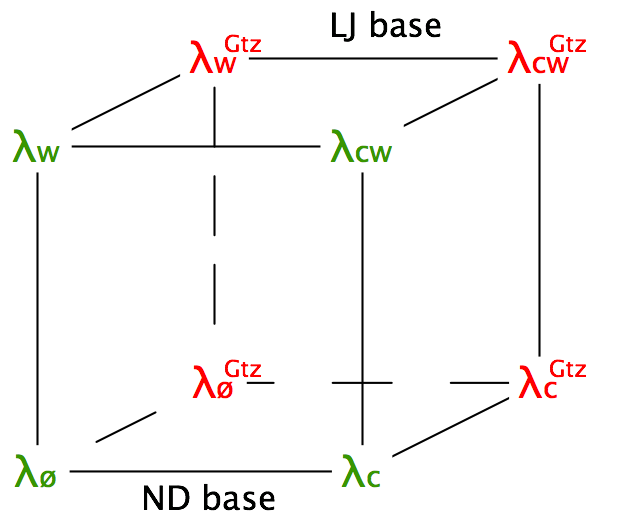}
\caption{Resource control cube}
\label{fig:cube}
\end{center}
\end{figure}

Let us turn our attention to type assignment systems for various term calculi.
The basic type assignment system is the one with \emph{simple types} in which the only forming operator is an arrow~$\rightarrow$.
Among many extensions of this type assignment system, the one that characterises the largest class of terms, is the extension with \emph{intersection types}, where a new type forming operator $\cap$ is introduced. The principal feature of this operator is the possibility to assign two different types to a certain term at the same time.
The intersection types have been originally introduced
in~\cite{coppdeza78,sall78,coppdeza80,pott80,coppdezavenn80} in order to characterise termination properties of various term calculi~\cite{bake92,gall98,ghil96}.
The extension of Curry--Howard correspondence to other formalisms brought the need for
intersection types into many different
settings~\cite{dougghillesc07,kikuchiRTA07,matthes2000,neer05}.

We introduce the intersection types into all eight calculi of the resource control cube.
Our intersection type assignment systems $\lR\cap$ and $\lGR\cap$ integrate intersections into logical rules, thus preserving syntax-directedness of the systems.
We assign a restricted form of intersection types to terms, namely strict
types, therefore eliminating the need for pre-order on types.
Using these intersection type assignment systems we manage to completely characterise strong normalisation in the cube by proving that
terms in all calculi in both bases enjoy the strong normalisation property if
and only if they are typeable.
While this characterization is well-known for the pure $\lambda$-calculus and $\lG$-calculus, it is a new result for the other six calculi of the cube.
The paper proves the characterization in a modular way, presenting one uniform proof for the first four calculi in natural deduction style, and one uniform proof for the four sequent style calculi.
Concerning the sequent calculus, to the best of our knowledge, there is no literature on explicit resource control operators for $\lG$-calculus, nor on the connection to intersection types (except for the initial version of this work presented in~\cite{ghilivetlesclika11}).

The rest of the paper is organised as follows.
In Section~\ref{sec:syntax} we introduce the untyped resource control cube
which consists of eight calculi $\lR$ and $\lGR$ with different combinations
of explicit/implicit control operators of weakening and contraction.
Basic type assignment systems with simple types for these calculi
are given in Section~\ref{sec:simple_types}.
Intersection type assignment systems with strict types are
introduced in Section~\ref{sec:int_types}.
By adapting the reducibility method to explicit resource control
operators, we first prove that typeable terms are strongly normalising
in $\lR$-calculi of the ND-base in Section~\ref{sec:typeSN}. Next, we prove
that typeability implies strong normalization in
$\lGR$-calculi of the LJ-base by using a combination of well-orders and a suitable
embedding of $\lGR$-terms into $\lR$-terms which preserves types
and enables the simulation of all the reductions in $\lGR$-calculi by the
operational semantics of the $\lR$-calculi.
Finally, in Section~\ref{sec:SNtypeBoth}, we prove
that strong normalisation implies typeability in both systems using
head subject expansion.
We conclude in Section~\ref{sec:conclusions}.

\tableofcontents

\section{Untyped resource control cube}
\label{sec:syntax}
\subsection{Resource control lambda calculi $\lR$}

In this section we present four calculi obtained by adding contraction and weakening to the
$\lambda$-calculus either as explicit or implicit operators. We
denote these four calculi by~$\lR$, where $\mathcal{R}\subseteq
\{\co, \w \}$ and $\co$, $\w$ denote explicit contraction and
weakening, respectively. \jel{The presented system operationally corresponds to the four calculi of the so-called implicit base of Kesner and Renaud's prismoid~\cite{kesnrena09}, as will be showed in Section~\ref{sec:vs}}. We use a notation along the lines of \cite{zunicPHD} and close to \cite{oost01}. It is slightly modified
w.r.t.~\cite{kesnrena09} in order to emphasize the correspondence
between these calculi and their sequent
counterparts.\footnote{Note that
in~\cite{kesnrena09}, instead of considering the sequent counterparts, the authors
consider the so called explicit base, i.e.\ the four corresponding calculi with explicit
substitutions.}  For the convenience of the reader and to avoid repetition, we present
all the calculi in a uniform way.  This implies that some
constructions or features are part of one calculus and not of the
others.  When a feature occurs in a calculus associated with the
operator~$\textsf{r} \in \mathcal{R}$  and is ignored elsewhere, we put this feature
between brackets indexed by the symbol~\textsf{r}.  For instance,
if we write $[x \in Fv(f)]_{\mathsf{w}}$, this means that the condition
$x \in Fv(f)$ appears only in the calculus which contains explicit
weakening, as seen from the index \textsf{w}.

In order to define the terms of the calculus, we introduce a
syntactic category called \emph{pre-terms}.
A \emph{pre-term} can
be a variable, an abstraction, an application, a contraction or a
weakening. The abstract syntax of the $\lR$ pre-terms is given by
the following:
$$
\begin{array}{lcrcl}
\textsf{Pre-terms}    &  & f & ::= & x\,|\,\lambda
x.f\,|\,ff \,|\, \cont{x}{x_1}{x_2}{f} \,|\,\weak{x}{f}\\
\end{array}
$$
where $x$ ranges over a denumerable set of term variables

\jel{The list of free variables of a pre-term $\f$, denoted by $Fv[\f]$,
is defined as follows (where $l,m$ denotes appending two lists and $l\setminus x$ denotes removing all occurrences  of an element from a list)}:
$$\begin{array}{c}
Fv[x] = x; \quad Fv[\lambda x.\f] = Fv[\f]\setminus x; \quad
Fv[fg] = Fv[f], Fv[g];\\
\left[Fv[\weak{x}{f}] = x,Fv[f]\right]_{\w}; \\
\left[Fv[\cont{x}{x_1}{x_2}{f}] = \left\{
    \begin{array}{rl}
        Fv[f], & x_1 \notin Fv[f]\;\; \mbox{and};\; x_2 \notin Fv[f]\\
        x, (Fv[f]\setminus \{x_1,x_2\}), & x_1 \in Fv[f] \;\;\mbox{or}\;\; x_2 \in Fv[f]
    \end{array}\right.
\right]_{\co}
\end{array}$$

\jel{The set of free variables, denoted by $Fv(f)$, is extracted out of the list $Fv[f]$, by un-ordering the list and removing multiple occurrences of each variable, if any. The set of bound variables of a pre-term $f$, denoted by $Bv(f)$, contains all variables of $f$ that are not free in it, i.e. $Bv(f)=Var(f)\setminus Fv(f)$.}
In $\cont{x}{x_1}{x_2}{f}$, the duplication binds the variables
$x_1$ and $x_2$ in $f$ and introduces a fresh variable $x$ if
at least one of $x_1,x_2$ is free in $f$. In $\weak {x}{f}$, the
variable $x$ is free. In order to avoid parentheses, we let the
scope of all binders extend to the right as much as possible.

The sets of $\lR$-terms, denoted by $\LR$, are subsets of the set of pre-terms
and are defined by the inference rules given in
Figure~\ref{fig:wf-small}. In the reminder of the paper, we will use $M, N, P...$ to denote terms.
We also use the notation $\weak{X}{M}$ for $\weak{x_1}{...\;\weak{x_n}{M}}$ and
$\cont{X}{Y}{Z}{M}$ for $\cont{x_1}{y_1}{z_1}{...\;\cont{x_n}{y_n}{z_n}{M}}$, where $X$,
$Y$ and $Z$ are lists of size $n$, consisting of all distinct variables $x_1,...,x_n,
y_1,..., y_n, z_1,...,z_n$. If $n=0$, i.e., if $X$ is the empty list, then $\weak{X}{M} =
\cont{X}{Y}{Z}{M} = M$.  In particular, given a term $N$, we use a specific notation:
\(\cont{Fv[N]}{Fv(N_1)}{FV{N_2}} M[N/x_1,N/x_2]\).  Assume that $N$ has the list
$Fv(N)=(z_1,...,z_p)$ as free variables.  We write $N$ as $N(z_1,...,z_p)$ and $N_1$ and
$N_2$, as the term in which the variables have been renamed with fresh variables.  I.e.,
$N_1= N(z_{1,1},... z_{1,p})$ and $N_2=
N(z_{2,1},... z_{2,p})$.  \[\cont{Fv[N]}{Fv(N_1)}{FV{N_2}} M[N/x_1,N/x_2]\]
is by definition:
\[\cont{z_1}{z_{1,1}}{z_{2,1}}...\cont{z_p}{z_{1,p}}{z_{2,p}}\,M[N(z_{1,1},... z_{1,p})/x_1][N(z_{2,1},... z_{2,p})/x_2].\]
Note that due to the equivalence relation defined in Figure~\ref{fig:equiv-lR}, we can use
these notations also for a set of triple of variables, all distinct. 

\begin{figure}[htpb]
\centerline{ \framebox{ $
    \begin{array}{c}
    \\
    \infer[\jel{(var)}]{x \in \LR}{}
    \\\\
      \begin{array}{c@{\qquad\qquad}c}
        \infer[\jel{(abs)}]{\lambda x.f \in \LR}
        {f \in \LR \;\; [x \in Fv(\f)]_{\w}} &
        \infer[\jel{(app)}]{fg \in \LR}
        {f \in \LR\;\; g \in \LR \;\; [Fv(f) \cap Fv(g) = \emptyset]_{\co}}
      \end{array}
      \\\\
      \begin{array}{c} 
      \infer[(\w \in \mathcal{R})\;\jel{(era)}]{\weak{x}{f} \in \LR}
      {f \in \LR \;\; x \notin Fv(f)} 
      \end{array}
      \\\\
      \begin{array}{c} 
      \infer[(\co \in \mathcal{R})\;\jel{(dup)}]{\cont{x}{x_1}{x_2}{f} \in \LR}
      {f \in \LR \;\; x_{1} \neq x_{2}\;\; x \notin Fv(f) \setminus \{x_{1}, x_{2}\} \;\; [x_1,x_2 \in Fv(f)]_{\w}}
    \end{array}\\ \\
  \end{array}
$ }} \caption{$\lR$-terms} \label{fig:wf-small}
\end{figure}

\begin{exa}
Pre-terms $\lambda x.y$ and $\cont{y}{y_1}{y_2}{x}$ are
$\lR$-terms only if $\w \notin \mathcal{R}$. Similarly,
pre-terms $\lambda x.xx$ and $\weak{x}{\lambda y.yy}$ are
$\lR$-terms only if $\co \notin \mathcal{R}$.
\end{exa}

In what follows we use Barendregt's convention~\cite{Bare84} for variables: in the same context a variable cannot be both free and bound. This applies to binders like $`l x.M$ which binds $x$ in $M$, $\cont{x}{x_1}{x_2}M$ which binds $x_1$ and $x_2$ in $M$, and also to the implicit substitution $M[N/x]$ which can be seen as a binder for $x$ in $M$.

Implicit substitution $M\isub{N}{x}$, \jel{where $x \notin Bv(M)$}, is defined in Figure~\ref{fig:sub-lR}.
In this definition, \jel{the following additional assumptions must hold}:
$$[Fv(M) \cap Fv(N) = \emptyset]_{\co},\quad \quad \jel{[x \in Fv(M)]_{\w}}$$
otherwise the substitution result would not be a well-formed term.
In the same definition, terms $N_1$ and $N_2$ are obtained from $N$ by renaming all free variables in $N$ by
fresh variables, and $M[N_1/x_1,N_2/x_2]$ denotes parallel substitution.\footnote{Note that the terms $N_{1}$ and $N_{2}$ do not have any free variables in common hence, it is not a problem to perform the substitution in parallel.} \jel{Notice that substitution is created only in $(\beta)$-reduction, yielding that $N$ is not any term, but the part of the term $(\lambda x.M)N$, therefore Barendregt's convention applies to it. For that reason we don't need side conditions like $y \notin Fv(N)$ in the definition of $(\lambda y.M)\isub{N}{x}$.}

\begin{figure}[hbtp]
\centerline{ \framebox{ $
\begin{array}{rclrcl}\\
x\isub{N}{x} & \triangleq & N  &
(\weak{y}{M})\isub{N}{x} &\triangleq  & \weak{\{y\} \setminus Fv(N)}{M\isub{N}{x}}, \;\;x \neq y\\
y\isub{N}{x} & \triangleq & y, \;\; x \not= y &
(\weak{x}{M})\isub{N}{x} & \triangleq  & \weak{Fv(N) \setminus Fv(M)}{M}\\
 (\lambda y.M)\isub{N}{x} & \triangleq  & \lambda y.M\isub{N}{x},
\;\;x \neq y
& (\cont{y}{y_1}{y_2}{M})\isub{N}{x} & \triangleq  &
\cont{y}{y_1}{y_2}{M\isub{N}{x}}, \;\;x \neq y\\
(MP)\isub{N}{x} & \triangleq  & M\isub{N}{x} P\isub{N}{x} &
(\cont{x}{x_1}{x_2}{M})\isub{N}{x} & \triangleq  &
\cont{Fv[N]}{Fv[N_1]}{Fv[N_2]}{M[N_1/x_1,N_2/x_2]}\\ \\
\end{array}
$ }}
\caption{Substitution in $\lR$-calculi} \label{fig:sub-lR}
\end{figure}

\jel{
In the following proposition, we prove that the substitution is well-defined.

\begin{prop}
If $M,N \in \LR$ and $x \notin Bv(M)$, then $M\isub{N}{x} \in \LR$, provided that 
$[Fv(M) \cap Fv(N) = \emptyset]_{\co}$ and $[x \in Fv(M)]_{\w}$.
\end{prop}
\begin{proof} Proposition if proved together with auxiliary statement:\\
If $M,N_1,N_2 \in \LR$, $\co \in \mathcal{R}$, $Fv(N_1) \cap Fv(N_2) = Fv(N_1) \cap Fv(M) = Fv(N_2) \cap Fv(M) = \emptyset$, $x_1 \notin N_2$ and $x_2 \notin N_1$, then $M[N_1/x_1,N_2/x_2] \in \LR$.
\end{proof}

The operational semantics for four calculi that form the ``natural deduction base'' of the
resource control cube are given in Figures~\ref{fig:red-lR} and \ref{fig:equiv-lR}.
Reduction rules of $\lR$-calculi are given in Figure~\ref{fig:red-lR}, whereas equivalences in are given in Figure~\ref{fig:equiv-lR}.
Reduction rules specific for each calculus are given in Figure~\ref{fig:ND-base}.

\begin{figure}[hbtp] \centerline{ \framebox{ $
\begin{array}{rrcl}\\
(\beta)             & (\lambda x.M)N & \rightarrow & M\isub{N}{x}  \\[2mm]
(\gamma_0)       & \cont{x}{x_1}{x_2}{y} & \rightarrow & y \quad y\not= x_1,x_2\\
(\gamma^{\;\;'}_0)       & \cont{x}{x_1}{x_2}{x_1} & \rightarrow & x \\
(\gamma_1)          & \cont{x}{x_1}{x_2}{(\lambda y.M)} &
\rightarrow & \lambda y.\cont{x}{x_1}{x_2}{M} \\
(\gamma_2)          & \cont{x}{x_1}{x_2}{(MN)} & \rightarrow &
(\cont{x}{x_1}{x_2}{M})N, \;\mbox{if} \; x_1,x_2 \notin Fv(N) \\
(\gamma_3)          & \cont{x}{x_1}{x_2}{(MN)} & \rightarrow &
M(\cont{x}{x_1}{x_2}{N}), \;\mbox{if} \; x_1,x_2 \notin Fv(M) \\
[2mm]
(\omega_1)          & \lambda x.(\weak{y}{M}) & \rightarrow & \weak{y}{(\lambda x.M)},\;x \neq y\\
(\omega_2)          & (\weak{x}{M})N & \rightarrow & \weak{\{x\}\setminus Fv(N)}{(MN)}\\
(\omega_3)          & M(\weak{x}{N}) & \rightarrow & \weak{\{x\}\setminus
    Fv(M)}{(MN)}\\[2mm]
(\gamma \omega_1)   & \cont{x}{x_1}{x_2}{(\weak{y}{M})} &
\rightarrow & \weak{y}{(\cont{x}{x_1}{x_2}{M})},\;y \neq x_1,x_2 \\
(\gamma \omega_2)   & \cont{x}{x_1}{x_2}{(\weak{x_1}{M})} & \rightarrow & M\isub{x}{x_2}\\ [2mm]
%
\\
\end{array}
$ }} \caption{Reduction rules of $\lR$-calculi} \label{fig:red-lR}
\end{figure}

\begin{figure}[hbtp]
\centerline{ \framebox{ $
\begin{array}{lrcl}\\
(`e_1)&\weak{x}{(\weak{y}{M})} & \equiv & \weak{y}{(\weak{x}{M})}\\
(`e_2)&\cont{x}{x_1}{x_2}{M} & \equiv & \cont{x}{x_2}{x_1}{M}\\
(`e_3)&\cont{x}{y}{z}{(\cont{y}{u}{v}{M})} & \equiv  &
\cont{x}{y}{u}{(\cont{y}{z}{v}{M})} \\
(`e_4)&\cont{x}{x_1}{x_2}{(\cont{y}{y_1}{y_2}{M})} & \equiv  &
\cont{y}{y_1}{y_2}{(\cont{x}{x_1}{x_2}{M})},\;\; x \neq y_1,y_2, \; y \neq x_1,x_2 \\ \\
\end{array}
$ }} \caption{Equivalences in $\lR$-calculi} \label{fig:equiv-lR}
\end{figure}

\begin{figure}[hbtp]
\begin{tabular}{c|c|c}
$\lR$-calculi & reduction rules & equivalences\\ \hline $\lo$ & $\beta$ \\
\hline  $\lc$ & $\beta$, $\gamma_0$, $\gamma^{\;\;'}_0$, $\gamma_1, \gamma_2, \gamma_3$ & $`e_2,`e_3,`e_4$\\
\hline  $\lw$ & $\beta$, $\omega_1, \omega_2, \omega_3$ & $`e_1$\\
\hline  $\lcw$ & $\beta$, $\gamma_1, \gamma_2, \gamma_3$, $\omega_1, \omega_2, \omega_3$,
$\gamma\omega_1$, $\gamma\omega_2$ & $`e_1,`e_2,`e_3,`e_4$
\end{tabular}
\caption{ND base of the resource control cube} \label{fig:ND-base}
\end{figure}

\subsubsection*{Which reductions in which system?}
\label{sec:pierr-reduct-which}

Generally speaking the reduction rules can \pierre{be} divided into four groups.
The main computational step is $\beta$ reduction.
The group of $(\gamma)$ reductions exists only in the two calculi
that contain contraction (i.e., if $\co \in \mathcal{R}$). These
rules perform propagation of contraction into the expression.
Similarly, $(\omega)$ reductions belong only to the two calculi
containing weakening (i.e., if $\w \in \mathcal{R}$). Their role is
to extract weakening out of expressions. This discipline allows us
to optimize the computation by delaying duplication of terms on
the one hand, and by performing erasure of terms as soon as
possible on the other.
Finally, the rules in
$(\gamma\omega)$ group explain the interaction between explicit
resource operators that are of different nature, hence these rules
exist only if $\mathcal{R} = \{\co,\w\}$.
The rules $(\gamma_0)$ and $(\gamma^{\;\;'}_0)$\footnote{The rules $(\gamma_0)$ and $(\gamma^{\;\;'}_0)$ correspond to the $\mathsf{CGc}$ rule in~\cite{kesnrena09}.} exist only if $\mathcal{R} = \{\co\}$
and they erase meaningless contractions.

\begin{rem}
Notice the asymmetry between the reduction rules of the calculi
$\lc$ and $\lw$, namely in $\lw$ there is no counterpart of the
$(\gamma_0)$ and $(\gamma^{\;\;'}_0)$ reductions of $\lc$. The reason can be tracked
back to the definition of $\lR$-terms in Figure
~\ref{fig:wf-small}, where the definition of the weakening
operator $\weak{x}{f}$ does not depend on the presence of the
explicit contraction. It would be possible to define weakening
where the condition $x \notin Fv(f)$ is not required if the
contraction is implicit. In that case, terms like
$\weak{x}{\weak{x}{M}}$ would exist, and the reduction
$(\omega_0): \weak{x}{M}  \rightarrow  M,  \;\mbox{if} \; x
\in Fv(M)$ would erase this redundant weakening. The typing rule
for this weakening would require multiset treatment of the bases
$\Gamma$, which is out of the scope of this paper.
\end{rem}
\pierre{\begin{exa}[$(`l x. x (\weak{x}{y}))z$]\label{ex:weak-sub} The notation
  $(\weak{x}{M})N "->"\weak{\{x\}\setminus Fv(N)}{MN}$ is a shorthand for two rules:
\begin{eqnarray*}
  (\weak{x}{M})N &"->"& \weak{x}{MN} \quad \textrm{if~} x`;FV(N)\\
  (\weak{x}{M})N &"->"& {MN} \quad \textrm{if~} x`:FV(N)\\
\end{eqnarray*}
Similarly $(\weak{y}{M})\isub{N}{x} \triangleq   \weak{\{y\} \setminus
  Fv(N)}{M\isub{N}{x}}$ is a shorthand for
\begin{eqnarray*}
  (\weak{y}{M})\isub{N}{x} &\triangleq  & \weak{y}{M\isub{N}{x}},  \textrm{if~} x`;FV(N)\\
  (\weak{y}{M})\isub{N}{x} &\triangleq  & {M\isub{N}{x}},  \textrm{if~} x`:FV(N)\\
\end{eqnarray*}

Let us illustrate the
subtlety of rules $(`w_1)$ and $(`w_2)$ dealing with $\odot$ on an example in which
$\mathsf{w}`:\mathcal{R}$. 
\begin{eqnarray*}
  (`l x. x (\weak{x}{y}))z &"-{^{`b}}>"& (x (\weak{x}{y}))[z/x]\\
  &=& z (\weak{z}{y})\\
  &"-{^{`w_3}}>"& z y.
\end{eqnarray*}
But we get also:
\begin{eqnarray*}
  (`l x. x (\weak{x}{y}))z &"-{^{`w_3}}>"& (`l . x y) z\\
&"-{^{`b}}>"& z y
\end{eqnarray*}
which is not surprising since $\lR$ is confluent.
\end{exa}
}

\subsection{The comparison of the resource control cube and the prismoid of resources}
\label{sec:vs}

In this subsection, we compare the two corresponding substructures of the resource control cube and the prismoid of resources, namely
$ND$-base of the cube and implicit base of the prismoid.

First, we will focus on differences between the particular elements of the two systems, and after that we will prove that they are operationally equivalent. For the sake of uniformity, we will use our notation to write also the terms of the prismoid.

The first difference is in the definition of free variables of the cube ($Fv(M)$), free variables of the prismoid ($fv(M)$) and positive
free variables of the prismoid ($fv^{+}(M)$). The relation among the three is $fv^{+}(M) \subseteq Fv(M) \subseteq fv(M)$, which will be illustrated by the following example.
\begin{exa} Let $M \equiv \cont{x}{x_1}{x_2}{\cont{x_1}{x_3}{x_4}{y}}$, which is a (well-formed) term when $\mathcal{R}=\{\co\}$. $Fv(M)=\{y\}$ and $fv^{+}(M)=\{y\}$, while $fv(M)=\{x,y\}$. Further, let $N \equiv \cont{x}{x_1}{x_2}{\cont{x_1}{x_3}{x_4}{\weak{x_2}{\weak{x_3}{\weak{x_4}{y}}}}}$, which is a (well-formed) term when $\mathcal{R}=\{\co,\w\}$. Now, $Fv(N)=\{x,y\}$ and $fv(N)=\{x,y\}$, while $fv^{+}(N)=\{y\}$.
\end{exa}

The next difference is in the notion of substitution of the cube ($M\isub{N}{x}$) and the corresponding notion in the prismoid's implicit base ($M\{N /x\}$). We will distinguish cases for the particular subclasses of terms.

\begin{itemize}
\item
In the case of $\lc$-terms, the substitutions differ when $M \equiv \cont{x}{x_1}{x_2}{M_1}$.\\
Let $M \equiv \cont{x}{x_1}{x_2}{y}$. Then, in the cube
\begin{center}
$(\cont{x}{x_1}{x_2}{y})\isub{N}{x}\;\triangleq \; \cont{Fv(N)}{Fv(N_1)}{Fv(N_2)}{y[N_1/x_1,N_2/x_2]}\;=\;\cont{Fv(N)}{Fv(N_1)}{Fv(N_2)}{y}$.
\end{center}
On the other side, in the prismoid, since $x \notin fv^{+}(\cont{x}{x_1}{x_2}{y})$ and $\w \notin \mathcal{B}$, we have:
\begin{center}
$(\cont{x}{x_1}{x_2}{y})\{N/x\}\;:=\;del_x(\cont{x}{x_1}{x_2}{y})\;=\;y$.
\end{center}
Since $\cont{Fv(N)}{Fv(N_1)}{Fv(N_2)}{y}$ reduces to $y$ by several $(\gamma_0)$ reductions, we conclude that in this case substitution in the cube is "smaller" then the one in the prismoid.\\
Now, let $M \equiv \cont{x}{x_1}{x_2}{x_1y}$. Then, in the cube
\begin{center}
$(\cont{x}{x_1}{x_2}{x_1y})\isub{N}{x}\;\triangleq \; \cont{Fv(N)}{Fv(N_1)}{Fv(N_2)}{(x_1y)[N_1/x_1,N_2/x_2]}\;=\;\cont{Fv(N)}{Fv(N_1)}{Fv(N_2)}{(N_1y)}$.
\end{center}
In the prismoid, since $x \in fv^{+}(\cont{x}{x_1}{x_2}{x_1y})$ (more precisely $|\cont{x}{x_1}{x_2}{x_1y}|_x^{+}=1$), we have:
\begin{center}
$(\cont{x}{x_1}{x_2}{x_1y})\{N/x\}\;:=\;del_x(\cont{x}{x_1}{x_2}{x_1y})\{\{N/x\}\}\;=\;\cont{fv(N)}{fv(N_1)}{fv(N_2)}{(x_1y){N_1/x_1}{N_2/x_2}} \;=\;\cont{fv(N)}{fv(N_1)}{fv(N_2)}{(N_1y)}$.
\end{center}
Now, since we already showed that $Fv(N) \subseteq fv(N)$, there are cases where the result of the substitution in the prismoid has more contractions then the corresponding one in the cube. To be completely precise, let $N \equiv \cont{z}{z'}{z''}{w}$ (yielding that $N_1 \equiv \cont{z_1}{z'}{z''}{w_1}$ and $N_2 \equiv \cont{z_2}{z'}{z''}{w_2}$). Now, in the cube $Fv(N)=\{w\}$, while in the prismoid $fv(N)=\{z,w\}$. So,
\begin{center}
$(\cont{x}{x_1}{x_2}{x_1y})\isub{\cont{z}{z'}{z''}{w}}{x}\;\triangleq \; \cont{w}{w_1}{w_2}{((\cont{z_1}{z'}{z''}{w_1})y)}$,
\end{center}
while on the other side
\begin{center}
$(\cont{x}{x_1}{x_2}{x_1y})\{N/x\}\;:=\;\cont{z}{z_1}{z_2}{\cont{w}{w_1}{w_2}{((\cont{z_1}{z'}{z''}{w_1})y)}}$.
\end{center}
But the latter reduces to the former term by one $CG_c$ reduction, hence in this case we conclude that substitution in the prismoid is "smaller" then the one in the cube.\\
\item In the case of $\lw$-terms, the substitutions differ when $M \equiv \lambda y.M_1$.\\
Let $|M_1|^{+}_x \geq 2$, for example $M \equiv \lambda y.((\weak{x}{y})x)x$. Then, in the cube:
\begin{center}
$(\lambda y.((\weak{x}{y})x)x)\isub{N}{x}\;\triangleq \; \lambda y.((\weak{Fv(N) \setminus Fv(y)}{y})N)N$,
\end{center}
while in the prismoid:
\begin{center}
$(\lambda y.((\weak{x}{y})x)x)\{N/x\}\;:=\;((\lambda y.((\weak{x}{y})z)x)\{N/z\})\{N/x\}\;=\;(\mathrm{del}_z(\lambda y.((\weak{x}{y})z)x)\{\{N/z\}\})\{N/x\}\;=\;(\lambda y.((\weak{x}{y})N)x)\{N/x\}\;=\mathrm{del}_x(\lambda y.((\weak{x}{y})N)x)\{\{N/x\}\}\;=(\lambda y.(yN)x)\{\{N/x\}\}\;=\;((\lambda y.y)N)N$.
\end{center}
So, in the prismoid, we do not substitute for the variable that is introduced by weakening (those non-positive occurrences of a variable are deleted during substitution execution, using $\mathrm{del}_x)$ operator). Therefore, the result of the substitution in the cube can have some extra weakenings comparing to the corresponding term in the prismoid. But, $\lambda y.((\weak{Fv(N) \setminus Fv(y)}{y})N)N$ reduces to $((\lambda y.y)N)N$ by a number of $(\omega_2)$ reductions\footnote{In the analogous case $M \equiv \lambda y.(x(\weak{x}{y}))x$ we would use $(\omega_3)$ reductions for the same purpose}. Again, we conclude that the substitution in the cube is "smaller" in this case.
\item In the case of $\lo$ and $\lcw$ terms substitutions do not differ.
\end{itemize}
From the analysis presented above, we can conclude that although substitution definitions in particular cases differ, two systems work equally when we look the whole operational semantics, i.e. substitution + reductions + equivalencies. 
}

\subsection{Resource control sequent lambda calculi $\lGR$}

An attempt of creating the sequent-style
$\lambda$-calculus (i.e., the system corresponding to the Gentzen's
$LJ$ system) was proposed by Barendregt and Ghilezan in~\cite{bareghil00} by keeping the original syntax of $\lambda$-calculus and changing the
type assignment rules in accordance with the inference rules of $LJ$. The obtained simply typed $\lambda LJ$-calculus,
although useful for giving a new simpler proof of the
Cut-elimination theorem, was not in one-to-one correspondence with $LJ$.

Herbelin realised that, in order to achieve such a correspondence,
significant changes in syntax should be made. In~\cite{herb95b},
he proposed $\overline{\lambda}$-calculus with explicit
substitution, a new syntactic construct called a list and new
operators for creating and manipulating the lists. The role of the
lists was to overcome the key difference between natural deduction
and sequent calculi, namely the associativity of applications.
While in ordinary $\lambda$-calculus applications are
left-associated, i.e., $(\lambda x.M)N_1N_2...N_k \equiv
((((\lambda x.M)N_1)N_2)...)N_k$, in $\overline{\lambda}$-calculus the application is
right-associated, i.e., $(\lambda x.M)[N_1,N_2,...,N_k] \equiv (\lambda
x.M)(N_1(N_2...(N_{k-1}N_k)))$. The $\overline{\lambda}$-calculus was the
first formal calculus whose simply typed version extended the
Curry-Howard correspondence to the intuitionistic sequent
calculus, more precisely, its cut-free restriction $LJT^{cf}$.

Relying on Herbelin's work, Esp\'{\i}rito Santo and Pinto
developed $\lambda^{Jm}$-calculus with generalised multiary
application~\cite{santpint03} and later $\lG$-calculus~\cite{jesTLCA07}, the calculus that fully corresponds to Gentzen's
$LJ$. In the $\lG$-calculus one can also distinguish two kinds of expressions,
namely terms and contexts, the later being the generalisation of
the lists from the $\overline{\lambda}$-calculus.

In this section we present the syntax and the operational
semantics of the four sequent
calculi with explicit or implicit resource control, denoted by
$\lGR$, where $\mathcal{R}\subseteq \{\co, \w \}$ and $\co$, $\w$
denote explicit contraction and weakening, respectively. These
four calculi are sequent counterparts of the four resource control
calculi $\lR$ presented above, and represent extensions of
Esp\'{\i}rito Santo's $\lG$-calculus.

The abstract syntax of the $\lGR$ pre-expressions is the
following:
$$
\begin{array}{lcrcl}
\textrm{Pre-terms}    &  & f & ::= & x\,|\,\lambda
x.f\,|\,fc \,|\, \cont{x}{x_1}{x_2}{f} \,|\,\weak{x}{f}\\
\textrm{Pre-contexts} & & c & ::= & \bindx
f\,|\,f::c\,|\,\weak{x}{c} \,|\, \cont{x}{x_1}{x_2}{c}
\end{array}
$$
where $x, x_1, \ldots, y, \ldots$ range over a denumerable set of term variables.

A \emph{pre-term} can be a variable, an abstraction, a cut (an
application), a contraction or a weakening. Note that the
application is of the form $fc$. A \emph{pre-context} is one of the following features:
a selection $\bindx f$, which turns a term into a
context by choosing an active variable; a context constructor
$f::c$ (usually called cons) which expands a context by introducing
a term into the left-most position; a weakening on a pre-context
or a contraction on a pre-context. Pre-terms and pre-contexts are
together referred to as \emph{pre-expressions} and will be ranged
over by $E$. Pre-contexts $\weak{x}{c}$ and
$\cont{x}{x_1}{x_2}{c}$ behave exactly like the corresponding
pre-terms $\weak{x}{f}$ and $\cont{x}{x_1}{x_2}{f}$ in the untyped
calculi, so they will not be treated separately.

The set of free variables of a pre-expression $E$, denoted by
$Fv(E)$, is defined as follows:
$$\begin{array}{c}
Fv(x) = x;\quad Fv(\lambda x.\f) = Fv(\f)\setminus \{x\}; \quad
Fv(fc) = Fv(f) \cup Fv(c);\\
Fv(\bindx f) = Fv(f)\setminus \{x\};\quad
Fv(f::c) = Fv(f) \cup Fv(c);\\
\left[Fv(\weak{x}{E}) = \{x\} \cup Fv(E)\right]_{\w}; \\
\left[Fv(\cont{x}{x_1}{x_2}{E}) = \left\{
    \begin{array}{rl}
        \jel{Fv(E)}, & \{x_1,x_2\} \cap Fv(E) = \emptyset\\
        \{x\} \cup Fv(E)\setminus \{x_1,x_2\}, & \{x_1,x_2\} \cap
        Fv(E) \neq \emptyset
    \end{array}\right.
\right]_{\co}.
\end{array}$$

The sets of $\lGR$-expressions $\LGR = \TGR \cup \KGR$ (where
$\TGR$ are the sets of $\lGR$-terms and $\KGR$ are the sets of $\lGR$-contexts) are the subsets of the set of pre-expressions, defined by
the inference rules given in Figure~\ref{fig:wf}.

We denote terms by
$t,u,v...$, contexts by $k,k',...$ and expressions by $e,e'$.

\begin{figure}[htpb]
\centerline{ \framebox{ $
    \begin{array}{c}
    \\
    \infer{x \in \TGR}{}
    \\\\
      \begin{array}{c@{\qquad\qquad}c}
        \infer{\lambda x.f \in \TGR}
        {f \in \TGR \;\; [x \in Fv(\f)]_{\w}} &
        \infer{fc \in \TGR}
        {f \in \TGR\;\; c \in \KGR \;\; [Fv(f) \cap Fv(c) = \emptyset]_{\co}}
      \end{array}
      \\\\
     \begin{array}{c@{\qquad}c}
      \infer{\bindx f \in \KGR}
      {f \in \TGR\;\;\;[x \in Fv(f)]_{\w}}&
      \infer{f::c \in \KGR}
      {f \in \TGR\;\;\;c \in \KGR\;\;\; [Fv(f)\cap Fv(c) = \emptyset]_{\co}}
    \end{array}
    \\\\
   \begin{array}{c@{\qquad}c}
      \infer[(\w \in \mathcal{R})]{\weak{x}{f} \in \TGR}
      {f \in \TGR \;\; x \notin Fv(f)}&
       \infer[(\w \in \mathcal{R})]{\weak{x}{c} \in \KGR}
      {c \in \KGR \;\; x \notin Fv(c)}
       \end{array}
      \\\\
      \begin{array}{c} 
      \infer[(\co \in \mathcal{R})]{\cont{x}{x_1}{x_2}{f} \in \TGR}
      {f \in \TGR \;\; x_{1} \neq x_{2}\;\;
      x \notin Fv(f) \setminus \{x_{1}, x_{2}\}  \;\; [x_1,x_2 \in Fv(f)]_{\w}}
    \end{array}
    \\\\
      \begin{array}{c} 
      \infer[(\co \in \mathcal{R})]{\cont{x}{x_1}{x_2}{c} \in \KGR}
      {c \in \KGR \;\; x_{1} \neq x_{2}\;\;
      x \notin Fv(c) \setminus \{x_{1}, x_{2}\}  \;\; [x_1,x_2 \in Fv(c)]_{\w}}\\ \\
    \end{array}
  \end{array}
$ }} \caption{$\lGR$-expressions}
\label{fig:wf}
\end{figure}

\begin{exa}
\begin{itemize}
\item[-] $\lambda x.x(y::\bindz z)$ belongs to all four sets $\LGo$, $\LGc$, $\LGw$ and $\LGcw$;
\item[-] $\lambda x.w(y::\bindz z)$ belongs to $\LGo$ and $\LGc$;
\item[-] $\lambda x.x(x::\bindz z)$ belongs to $\LGo$ and $\LGw$;
\item[-] $\lambda x.y(y::\bindz z)$ belongs only to the $\LGo$;
\item[-] $\lambda x.\weak{x}{y(y::\bindz z)}$ belongs only to $\LGw$;
\item[-] $\lambda x.\cont{y}{y_1}{y_2}{y_1(y_2::\bindz z)}$ belongs only to
$\LGc$;
\item[-] $\lambda x.\weak{x}{\cont{y}{y_1}{y_2}{y_1(y_2::\bindz z)}}$
belongs only to $\LGcw$.
\end{itemize}
\end{exa}

The inductive definition  of the meta operator of implicit substitution $e\isub{t}{x}$, representing the substitution of free variables, is given in Figure~\ref{fig:sub-lGR}.
In this definition, in case $\co \in \mathcal{R}$, the following condition must be satisfied:
$$Fv(e) \cap Fv(t) = \emptyset,$$ otherwise the substitution result would not be a well-formed term.
In the same definition, terms $t_1$ and $t_2$ are obtained from $t$ by renaming all free variables in
$t$ by fresh variables.

\begin{figure}[hbtp]
\centerline{ \framebox{ $
\begin{array}{rclrcl}\\
x\isub{t}{x} & \triangleq & t &
y\isub{t}{x} & \triangleq & y\\
(\lambda y.v)\isub{t}{x} & \triangleq  & \lambda y.v\isub{t}{x},
\;\;x \neq y
& (\weak{y}{e})\isub{t}{x} &\triangleq  & \weak{\{y\} \setminus Fv(t)}{e\isub{t}{x}}, \;\;x \neq y \\
(vk)\isub{t}{x} & \triangleq  & v\isub{t}{x}\, k\isub{t}{x}&
(\weak{x}{e})\isub{t}{x} & \triangleq  & \weak{Fv(t) \setminus Fv(e)}{e}\\
(v::k)\isub{t}{x}& \triangleq &v\isub{t}{x}::k\isub{t}{x}&
(\cont{y}{y_1}{y_2}{e})\isub{t}{x} & \triangleq  &
\cont{y}{y_1}{y_2}{e\isub{t}{x}}, \;\;x \neq y\\
%
(\bindy v)\isub{t}{x} & \triangleq & \bindy v\isub{t}{x} &
(\cont{x}{x_1}{x_2}{e})\isub{t}{x} & \triangleq  &
\cont{Fv(t)}{Fv(t_1)}{Fv(t_2)}{e\isub{t_1}{x_1}\isub{t_2}{x_2}}\\
\\
\end{array}
$ }} \caption{Substitution in $\lGR$-calculi} \label{fig:sub-lGR}
\end{figure}


The computation over the set of $\lGR$-expressions reflects the
cut-elimination process, and manages the explicit/implicit
resource control operators. Four groups of reductions in
$\lGR$-calculi are given in Figure~\ref{fig:red-lGR}, while
the equivalences are given in Figure~\ref{fig:equiv-lGR}.
Reduction rules and equivalences specific to each of the
four term calculi forming the ``LJ base'' of the resource control
cube are given in Figure~\ref{fig:LJ-base}.

\begin{figure}[hbtp]
\centerline{ \framebox{ $
\begin{array}{rrcl}\\
(\beta) & (\lambda x.t)(u::k) & \rightarrow & u(\bindx tk) \\
(\sigma) & t(\bindx v)              & \rightarrow & v\isub{t}{x}\\
(\pi) & (tk)k'                    & \rightarrow & t(\append{k}{k'})\\
(\mu) & \bindx xk             & \rightarrow & k\\[2mm]
(\gamma_0)       & \cont{x}{x_1}{x_2}{y} & \rightarrow & y \quad y\not= x_1,x_2\\
(\gamma^{\;\;'}_0)       & \cont{x}{x_1}{x_2}{x_1} & \rightarrow & x \\
(\gamma_1)       & \cont{x}{x_1}{x_2}{(\lambda y.t)} & \rightarrow
& \lambda y.\cont{x}{x_1}{x_2}{t}\\
(\gamma_2)       & \cont{x}{x_1}{x_2}{(tk)} & \rightarrow & (\cont{x}{x_1}{x_2}{t})k, \;\;\mbox{if} \; x_1,x_2 \notin Fv(k)\\
(\gamma_3)       & \cont{x}{x_1}{x_2}{(tk)} & \rightarrow & t(\cont{x}{x_1}{x_2}{k}), \;\;\mbox{if} \; x_1,x_2 \notin Fv(t)\\
(\gamma_4)       & \cont{x}{x_1}{x_2}{(\bindy t)} & \rightarrow & \bindy (\cont{x}{x_1}{x_2}{t})\\
(\gamma_5)       & \cont{x}{x_1}{x_2}{(t::k)} & \rightarrow & (\cont{x}{x_1}{x_2}{t})::k, \;\;\mbox{if} \; x_1,x_2 \notin Fv(k)\\
(\gamma_6)       & \cont{x}{x_1}{x_2}{(t::k)} & \rightarrow & t::(\cont{x}{x_1}{x_2}{k}), \;\;\mbox{if} \; x_1,x_2 \notin Fv(t)\\[2mm]
(\omega_1)       & \lambda x.(\weak{y}{t}) & \rightarrow & \weak{y}{(\lambda x.t)},\;\;x \neq y\\
(\omega_2)       & (\weak{x}{t})k & \rightarrow & \weak{\{x\}\setminus Fv(k)}{(tk)}\\
(\omega_3)       & t(\weak{x}{k}) & \rightarrow & \weak{\{x\}\setminus Fv(t)}{(tk)}\\
(\omega_4)       & \bindx (\weak{y}{t}) & \rightarrow & \weak{y}{(\bindx t)},\;\;x \neq y\\
(\omega_5)       & (\weak{x}{t})::k & \rightarrow & \weak{\{x\}\setminus Fv(k)}{(t::k)}\\
(\omega_6)       & t::(\weak{x}{k}) & \rightarrow & \weak{\{x\}\setminus Fv(t)}{(t::k)}\\[2mm]
(\gamma \omega_1)       & \cont{x}{x_1}{x_2}{(\weak{y}{e})} &
\rightarrow & \weak{y}{(\cont{x}{x_1}{x_2}{e})}  \qquad x_1\neq y
\neq x_2 \\
(\gamma \omega_2)       & \cont{x}{x_1}{x_2}{(\weak{x_1}{e})} &
\rightarrow & e\isub{x}{x_{2}}
\\ \\
\end{array}
$ }} \caption{Reduction rules of $\lGR$-calculi}
\label{fig:red-lGR}
\end{figure}

\begin{figure}[hbtp]
\centerline{ \framebox{ $
\begin{array}{lrcl}\\
(`e_1) & \weak{x}{(\weak{y}{e})} & \equiv & \weak{y}{(\weak{x}{e})} \\
(`e_2) & \cont{x}{x_1}{x_2}{e} & \equiv & \cont{x}{x_2}{x_1}{e}\\
(`e_3) & \cont{x}{y}{z}{(\cont{y}{u}{v}{e})} & \equiv  &
    \cont{x}{y}{u}{(\cont{y}{z}{v}{e})} \\
(`e_4) & \cont{x}{x_1}{x_2}{(\cont{y}{y_1}{y_2}{e})} &
    \equiv & \cont{y}{y_1}{y_2}{(\cont{x}{x_1}{x_2}{e})},\;\; x \neq y_1,y_2, \; y \neq x_1,x_2 \\ \\
\end{array}
$ }} \caption{Equivalences in $\lGR$-calculi}
\label{fig:equiv-lGR}
\end{figure}

\begin{figure}[hbtp]
\begin{tabular}{c|c|c}
$\lGR$-calculi & reduction rules  & \jel{equivalences}\\ \hline $\lGo$ & $\beta$, $\pi$, $\sigma$, $\mu$  \\
\hline  $\lGc$ & $\beta$, $\pi$, $\sigma$, $\mu$, $\gamma_0$ - $\gamma_6$ & $`e_2,`e_3,`e_4$\\
\hline  $\lGw$ & $\beta$, $\pi$, $\sigma$, $\mu$, $\omega_1$ - $\omega_6$ & $`e_1$\\
\hline  $\lGcw$ & $\beta$, $\pi$, $\sigma$, $\mu$, $\gamma_1$ -
$\gamma_6$, $\omega_1$ - $\omega_6$, $\gamma\omega_1$,
$\gamma\omega_2$ & $`e_1$ - $`e_4 $
\end{tabular}
\caption{LJ base of the resource control cube} \label{fig:LJ-base}
\end{figure}

\pierre{The first group consists of $(\beta)$, $(\pi)$, $(\sigma)$ and
$(\mu)$ reductions that exist in all four $\lGR$-calculi.  
$(\beta)$~reduction is the main computational step. In particular ($`b$) 
creates a potential substitution, which is made actual by $(\sigma)$. In that sense,
substitutions are controlled (i.e. can be 
delayed) although they are implicit. Note that this feature is
not present in $\lR$-calculi since it is a consequence of the
existence of contexts. For more details see~\cite{jesRTA07}.
Combination $(\beta) + (\sigma)$ is the
traditional $(\beta)$ in $\lambda$-calculus. 
Rule $(\pi)$ simplifies the head of a cut ($t$ is the head of
$t k$). Rule $(\mu)$ erases the sequence made of a trivial cut
(a cut is trivial if its head is a variable) followed by the
selection of the same variable. ($`m$) is therefore a
kind of garbage collection.

In $(\pi)$ rule, the meta-operator $@$, called \emph{append},
joins two contexts and is defined as:
$$
\begin{array}{rclcrcl}
       \app{(\bindx t)}{k'} & = & \bindx tk'                        & \qquad & \app{(u::k)}{k'} & = & u::(\app{k}{k'})\\
\app{(\weak{x}{k})}{k'} & = & \weak{x}{(\app{k}{k'})} & \qquad &
\app{(\cont{x}{y}{z}{k})}{k'} & = & \cont{x}{y}{z}{(\app{k}{k'})}.
\end{array}
$$

If $\co \in \mathcal{R}$,
the group $(\gamma)$ in $\lGR$ has  three new reductions
which handle the interaction between contraction and selection and context
construction.  If $\w \in \mathcal{R}$, the group $(\omega)$ in $\lGR$   has three new
reductions which handle the interaction between weakening and selection and context construction.
Finally, the group  $(\gamma \omega)$ in $\lGR$ has two new rules which handle
in contexts the interaction between explicit resource operators  of several nature.

We have presented the syntax and the reduction rules of $\lGR$, $\mathcal{R}\subseteq
\{\co, \w \}$, a family of four intuitionistic sequent term calculi obtained by
instantiating $\mathcal{R}$.  $\mathcal{R} = \emptyset$ gives the well-known
\emph{lambda Gentzen calculus} $\lG$, proposed by Esp\'{\i}rito
Santo~\cite{jesTLCA07}, whose simply typed version corresponds to the intuitionistic
sequent calculus with cut and implicit structural rules, through the Curry-Howard
correspondence. $\mathcal{R} = \{\co, \w \}$, gives the \emph{resource control lambda
  Gentzen calculus}, $\llG$, whose call-by-value version was proposed and
investigated in ~\cite{ghilivetlesczuni11}.  Simply typed $\llG$ extends the
Curry-Howard correspondence to the intuitionistic sequent calculus with explicit
structural rules, namely weakening and contraction. Finally, $\mathcal{R} = \{\co \}$
and $\mathcal{R} = \{ \w \}$ give two new calculi, namely $\lGc$ and $\lGw$. Those
calculi could be related to substructural logics, as this will be elaborated in the
sequel.}

\pierre{
\subsection{Strong normalisation and substitution}
\label{sec:sn_subs}

Since the definition of substitution in $\lR$ and in $\lGR$ is not classical, it is
worth proving that if a term in which other terms have been substituted is strongly
normalising, the term itself is strongly normalising.  
First let us prove a substitution lemma.  Before we need another lemma on
substitutions.
\begin{lem}\label{lem:xnotfree}
  If $x`;Fv(M)$ then $M\isub{N}{x} =M$.
\end{lem}
\begin{proof}
  The proof woks by induction after noticing that only definitions:
 \begin{eqnarray*}
   (\weak{y}{M})\isub{N}{x} &\triangleq  & \weak{\{y\} \setminus
     Fv(N)}{M\isub{N}{x}}, \;\;x \neq y\\
 (\cont{y}{y_1}{y_2}{M})\isub{N}{x} & \triangleq  &
\cont{y}{y_1}{y_2}{M\isub{N}{x}}, \;\;x \neq y
 \end{eqnarray*}
apply for weakening and contraction.
\end{proof}
\begin{lem}[Substitution Lemma]
  $M\isub{N}{x}\isub{P}{y} = M\isub{P}{y}\isub{N\isub{P}{y}}{x}$.
\end{lem}
\begin{proof}
  Notice first the traditional convention that $x`;Fv(N)$, $x`;Fv(P)$ and $y`;Fv(P)$,
  which plays obviously an essential role in the proof.

The proof is by induction on the structure of the terms and most of the cases are
classical in $`l$-calculus. We are going only to consider the cases which are
specific to $\lR$.
\begin{itemize}
\item $M`=\weak{z}{Q}$ with $x$, $y$ and $z$ all different and $z`:Fv(N)$.  Then
  \begin{eqnarray*}
  (\weak{z}{Q})\isub{N}{x}\isub{P}{y} &=& Q\isub{N}{x}\isub{P}{y} \\
  &=&  Q\isub{P}{y}\isub{N\isub{P}{y}}{x}  \qquad \textrm{by induction}\\
&=&\weak{z}{(Q\isub{P}{y})}\isub{N\isub{P}{y}}{x}\\
&=& (\weak{z}{Q})\isub{P}{y}\isub{N\isub{P}{y}}{x}
\end{eqnarray*}
\item $M`=\weak{z}{Q}$ with $x$, $y$ and $z$ all different, $z`;Fv(N)$ and
  $z`:Fv(P)$. Then 
 \begin{eqnarray*}
  (\weak{z}{Q})\isub{N}{x}\isub{P}{y} &=& (\weak{z}{Q}\isub{N}{x})\isub{P}{y} \\
&=&Q\isub{N}{x}\isub{P}{y}\\
  &=&  Q\isub{P}{y}\isub{N\isub{P}{y}}{x}\\
&=&\weak{z}{(Q\isub{P}{y})}\isub{N\isub{P}{y}}{x} \qquad \textrm{by induction}\\
&=& (\weak{z}{Q})\isub{P}{y}\isub{N\isub{P}{y}}{x}
\end{eqnarray*}
\item $M`=\weak{z}{Q}$ with $x$, $y$ and $z$ all different, $z`;Fv(N)$ and
  $z`;Fv(P)$. Then 
\begin{eqnarray*}
  (\weak{z}{Q})\isub{N}{x}\isub{P}{y} &=& (\weak{z}{Q}\isub{N}{x})\isub{P}{y} \\
&=&\weak{z}{(Q\isub{N}{x}\isub{P}{y})}\\
  &=&\weak{z}{(Q \isub{P}{y}\isub{N\isub{P}{y}}{x})}\\
&=&\weak{z}{(Q\isub{P}{y})}\isub{N\isub{P}{y}}{x} \qquad \textrm{by induction}\\
&=& (\weak{z}{Q})\isub{P}{y}\isub{N\isub{P}{y}}{x}
\end{eqnarray*}
\item $M`=\weak{x}{Q}$.
  \begin{eqnarray*}
   (\weak{x}{Q})\isub{P}{y}\isub{N\isub{P}{y}}{x} &=&
   (\weak{x}{Q}\isub{P}{y})\isub{N\isub{P}{y}}{x} \\
   &=& \weak{(Fv(N\isub{P}{y}) \setminus Fv(Q\isub{P}{y}))}{(Q\isub{P}{y})}\\
   &=& \weak{(Fv(N)\cup Fv(P) \setminus y \setminus Fv(Q)\cup Fv(P) \setminus y)}{(Q\isub{P}{y})}\\
   &=& \weak{Fv(N) \setminus Fv(Q)}{(Q\isub{P}{y})}\\
   &=& (\weak{Fv(N) \setminus Fv(Q)}{Q})\isub{P}{y}\qquad \textrm{because~}
   x`;Fv(N)\\
   &=& (\weak{x}{Q})\isub{N}{x}\isub{P}{y}
  \end{eqnarray*}
\item $M`=\weak{y}{Q}$.
  \begin{eqnarray*}
       (\weak{y}{Q})\isub{P}{y}\isub{N\isub{P}{y}}{x}
       &=&(\weak{(Fv(P)\setminus Fv(Q))}{Q})\isub{N\isub{P}{y}}{x}\\
       &=& \weak{(Fv(P)\setminus Fv(Q))}{(Q\isub{N\isub{P}{y}}{x})}\\
       &=& \weak{(Fv(P)\setminus Fv(Q))}{(Q\isub{P}{y}\isub{N\isub{P}{y}}{x})} \qquad
       \textrm{by Lemma~\ref{lem:xnotfree} since~} y`;Fv(Q)\\
       &=& \weak{(Fv(P)\setminus Fv(Q))}{(Q\isub{N}{x}\isub{P}{y})} \qquad \textrm{by
         induction}\\
       &=& (\weak{y}{Q})\isub{N}{x}\isub{P}{y}
  \end{eqnarray*}
\item $M`=\cont{z}{z_1}{z_2}Q$ with $z$, $x$ and $y$ all different and different of
  $z_1$ and $z_2$ by convention on bound variables.
  \begin{eqnarray*}
    (\cont{z}{z_1}{z_2}Q)\isub{P}{y}\isub{N\isub{P}{y}}{x} &=&
    (\cont{z}{z_1}{z_2}Q\isub{P}{y})\isub{N\isub{P}{y}}{x}\\
    &=&  \cont{z}{z_1}{z_2}(Q\isub{P}{y}\isub{N\isub{P}{y}}{x})\\
    &=& \cont{z}{z_1}{z_2}(Q\isub{N}{x}\isub{P}{y}) \qquad \textrm{by induction}\\
  &=& (\cont{z}{z_1}{z_2}Q\isub{N}{x})\isub{P}{y}\\
  &=& (\cont{z}{z_1}{z_2}Q)\isub{N}{x}\isub{P}{y}.
  \end{eqnarray*}
\item $M`=\cont{x}{x_1}{x_2}Q$.
  \begin{eqnarray*}
    (\cont{x}{x_1}{x_2}Q)\isub{N}{x}\isub{P}{y} &=&
    (\cont{Fv(N)}{Fv(N_1)}{Fv(N_2)}Q)\isub{N_1}{x_2}\isub{N_2}{x_2})\isub{P}{y}\\ 
    &=&
    \cont{Fv(N)}{Fv(N_1)}{Fv(N_2)}(Q\isub{N_1}{x_1}\isub{N_2}{x_2})\isub{P}{y})\\ 
    &=&
    \cont{Fv(N)}{Fv(N_1)}{Fv(N_2)}(Q\isub{P}{y}\isub{N_1\isub{P}{y}}{x_1}\isub{N_2\isub{P}{y}}{x_2}) 
    \qquad \textrm{by induction} \\ 
    &=& (\cont{x}{x_1}{x_2}Q)\isub{P}{y}\isub{N\isub{P}{y}}{x}.
  \end{eqnarray*}
\item $M`=\cont{y}{y_1}{y_2}Q$.
  \begin{eqnarray*}
    (\cont{y}{y_1}{y_2}Q)\isub{N}{x}\isub{P}{y} &=&
  (\cont{y}{y_1}{y_2}Q\isub{N}{x})\isub{P}{y} \\
     &=&
    \cont{Fv(P)}{Fv(P_1)}{Fv(P_2)}(Q\isub{N}{x}\isub{P_1}{y_1}\isub{P_2}{y_2})\\ 
    &=&
   \cont{Fv(P)}{Fv(P_1)}{Fv(P_2)}(Q\isub{P_1}{y_1}\isub{P_2}{y_2}\isub{N\isub{P_1}{y_1}\isub{P_2}{y_2}}{x})\\ 
    && \qquad \qquad \textrm{by induction} \\ 
    &=& (\cont{y}{y_1}{y_2}Q)\isub{P}{y}\isub{N\isub{P}{y}}{y}.
  \end{eqnarray*}
\end{itemize}

\end{proof}

Recall the concept of \emph{strong normalisation}.

\begin{defi}
A $\lR$-term $M$ is called \emph{strongly normalising} if and only if all reduction sequences starting with $M$ terminate.  By a
  \emph{reduction sequence} we mean a sequence of terms \sil{$(M_i)_{i\ge
    0}$ such that $M_i \rightarrow M_{i+1}$ or $M_i `= M_{i+1}$.}
We denote the set of all strongly normalising $\lR$-terms with $\SNR$ and the set of all closed strongly normalising $\lR$-terms with $\SNc$.
\end{defi}

\begin{prop}
  If $M\isub{N}{x}`:\SNR$  then $M`:\SNR$.
\end{prop}
\begin{proof}
  The proof works by induction on the lexicographic product $"-{^{\lR}}>" `*
  \supseteq$ i.e., on the reduction by $\lR$ followed by the subterm relation.  More
  precisely given a $M$ assume $M\isub{N}{x}`:\SNR$ and assume the lemma holds when
  the term is a contracted of $M\isub{N}{x}`:\SNR$ or a subterm of $M$ and prove that
  $M`:\SNR$.
  \begin{itemize}
  \item If $M`= x$ or $M`=y$, the result is trivial.
  \item If $M`= `l y . P$, then $`l y. (P[N/x])`:\SNR$, therefore $P[N/x]`:\SNR$ and
    by induction $P`:\SNR$ and since abstraction creates no redex, $`l y.P`:\SNR$.
  \item If $M `= (M P)[N/x]$. Since $(MP)\isub{N}{x} = M\isub{N}{x}\,P\isub{N}{x}$,
    we get that $M`:\SNR$ and ${P`:\SNR}$ by induction.  If $M\not`=`l y.Q$,
    $M\,P`:\SNR$.  If $M`=`l y.Q$.  We know that \[(`l y.Q) P \isub{N}{x} = (`l
    y. Q\isub{N}{x})\,P\isub{N}{x}\] which reduces to
    $Q\isub{N}{x}\isub{P\isub{N}{x}}{y}$, which is equal to
    $Q\isub{P}{y}\isub{N}{x}$, by the substitution lemma and which in $\SNR$ by
    induction.
  \item If $M`= \weak{y}{P}$ and $y`:Fv(N)$, then $(\weak{y}{P})\isub{N}{x} =
    P\isub{N}{x}$ (see discussion in Example~\ref{ex:weak-sub}).  By induction
    $P`:\SNR$. Since no rule applies on the top of $\weak{y}{P}$, we conclude that
    $\weak{y}{P}`:\SNR$.
 \item If $M`= \weak{y}{P}$ and $y`;Fv(N)$, then $(\weak{y}{P})\isub{N}{x} =
    \weak{y}{(P\isub{N}{x})}$. 
    $\weak{y}{(P\isub{N}{x})}`:\SNR$ implies that $P\isub{N}{x}$, then by induction
    $P`:\SNR$ and since no rule applies on the top of $\weak{y}{P}$, we conclude that
    $\weak{y}{P}`:\SNR$.
  \item If $M`= \weak{x}{P}$, then $(\weak{x}{P})\isub{N}{x} = \weak{Fv(N) \setminus
      Fv(P)}{P}$.  If $\weak{Fv(N) \setminus Fv(P)}{P}`:\SNR$ then $P`:\SNR$, and
    like before, since no rule applies on the top of $\weak{x}{P}$, we conclude that
    ${\weak{x}{P}`:\SNR}$.
  \item If $M`=\cont{y}{y_1}{y_2}{P}$ with $x\neq y$, then
    $\cont{y}{y_1}{y_2}{(P\isub{N}{x})}`:\SNR$ and $P\isub{N}{x}`:\SNR$ as well.  By
    induction $P`:\SNR$.  To conclude that  $\cont{y}{y_1}{y_2}{P}`:\SNR$ we have to
    analyse all the cases where one of the rules $(`g)$ applies. Consider only two
    cases:
    \begin{itemize}
    \item $P`=R S$ and $x_1,x_2`;FV(S)$, then $(`g_2)$ applies.  Therefore
      $M\isub{N}{x}`=\cont{y}{y_1}{y_2}{(R\isub{N}{x} S\isub{N}{x})}`:\SNR$. On the
      top $\cont{y}{y_1}{y_2}{(R\isub{N}{x} S\isub{N}{x})}$ reduces to
      $(\cont{y}{y_1}{y_2}{R\isub{N}{x}})\, S\isub{N}{x} = (\cont{y}{y_1}{y_2}{R})
      \isub{N}{x}\,  S\isub{N}{x} = ((\cont{y}{y_1}{y_2}{R})
      S)\isub{N}{x}$.  By induction $(\cont{y}{y_1}{y_2}{R})
      S`:\SNR$. Hence $`:\SNR$ since all its reducts are in $\SNR$.
    \item $P`=\weak{y_1}{R}$ then $(`g`w_2)$ applies. Therefore
      $M\isub{N}{x}`=\cont{y}{y_1}{y_2}{\weak{y_1}{R}\isub{N}{x}}`:\SNR$.  Then
      ${R}\isub{N}{x}`:\SNR$ and by induction $R`:\SNR$  Let us prove that the head
      reduce of $M = \cont{y}{y_1}{y_2}{\weak{y_1}{R}}$ is in $\SNR$. By  $(`g`w_2)$
        we get $R\isub{y}{y_2}$ which is in $\SNR$ as just a renaming of $R$.
    \end{itemize}

  \end{itemize}
\end{proof}
}

\section{Simple types for resource control cube}
\label{sec:simple_types}

In this section we summarise the type assignment systems that
assign \emph{simple types} to all eight calculi of the resource
control cube in a uniform way. Simple types for $\lG$-calculus
were introduced by Esp\'{\i}rito Santo in~\cite{jesTLCA07}. As far
as resource control calculi are concerned, simple types were
introduced to $\llxr$-calculus by Kesner and Lengrand
in~\cite{kesnleng07} and to resource control lambda Gentzen
calculus $\llG$  in~\cite{ghilivetlesczuni11}.
The syntax of types is defined as follows:
$$
\textrm{Simple Types}  \qquad \tA, \tB \;\;  ::= \;\; p \mid \tA \to \tB
$$
\noindent where $p$ ranges over an enumerable set of type atoms
and $\tA \to \tB$ are arrow types.
We denote types \silv{by} $\tA,\tB,\tC...$.

\begin{defi}
\label{def:typass-basis}
\rule{0in}{0in}
\begin{itemize}
    \item[(i)]
    A basic type assignment is an expression of the form $x:\tA$, where
    $x$ is a term variable and $\tA$ is a type.
    \item[(ii)]
    A basis $\Gamma$ is a set $\{x_{1} : \tA_{1}, \ldots, x_{n} : \tA_{n}\}$
    of basic type assignments, where all term variables are different.
    $Dom(\Gamma)=\{x_{1}, \ldots, x_{n}\}$.
    \item[(iii)]
     A basis extension $\Gamma, x : \tA$ denotes the set $\Gamma \cup \{x:\tA\}$, where
    $x \not \in Dom(\Gamma).$ $\Gamma,\Delta$ represents the
    disjoint union of the two bases $\Gamma$ and $\Delta$.
    \item[(iv)]
    $\Gamma \cup_{\co} \Delta$ denotes \silv{the} standard union of the bases,
    if $\co \notin \mathcal{R}$, and \silv{the} disjoint union, if $\co \in
    \mathcal{R}$.
\end{itemize}
\end{defi}


The type assignment systems $\lR \! \! \! \to$ \silv{for the natural
  deduction base (ND-base) of the resource control cube} are given in
Figure~\ref{fig:simptyp-lR}.

\begin{figure}[hbtp]
\centerline{ \framebox{
$
\begin{array}{c}
\\
\infer[(Ax_{iw})]{\Gamma, x:\tA \vdashr x:\tA}{\w
\notin\mathcal{R}}\quad\quad \infer[(Ax_{ew})]{x:\tA \vdashr
x:\tA}{\w \in \mathcal{R}}\\ \\
\infer[(\to_R)]{\Gamma\vdashr \lambda x.M:\tA\to \tB}
                        {\Gamma, x:\tA \vdashr M:\tB} \quad\quad
\infer[(\to_{E})]{\Gamma \cup_{\co} \Delta \vdashr MN:\tB}
                    {\Gamma \vdashr M:\tA \to \tB & \Delta \vdashr N:\tB}\\ \\
\infer[(Cont)]{\Gamma, z:\tA \vdashr \cont{z}{x}{y}{M}:\tB}
                    {\Gamma, x:\tA,y:\tA \vdashr M:\tB & \co \in \mathcal{R}} \quad\quad
\infer[(Weak)]{\Gamma, x:\tA \vdashr \weak{x}{M}:\tB}
                    {\Gamma \vdashr M:\tB & \w \in \mathcal{R}}\\ \\
\end{array}
$ }} \caption{$\lR  \! \! \! \to$: Simply typed $\lR$-calculi} \label{fig:simptyp-lR}
\end{figure}



The type assignment systems $\lGR \! \! \! \to$ \silv{for the sequent
  base (LJ-base) of the resource control cube} are given in
Figure~\ref{fig:simptyp-lGR}.  In $\lGR$, we distinguish two sorts of
\jel{type assignments}:
\begin{itemize}
\item[-]
$\Gamma \vdash t:\tA$ for typing a term and 
\item[-] 
$\Gamma;\tB \vdash k:\tA$, 
a \jel{type assignment} 
for typing a context, with a \emph{stoup}. $`b$ is called the stoup.
\end{itemize} 
\pierre{The stoup
is a specific place for the last base assignment. It occurs after a
semi-colon. If we consider that the computation  is implemented by $(`b)$ and $(`s)$,
the formula in the stoup is where the
computation actually takes place.}

\begin{figure}[hbtp]
\centerline{ \framebox{
$
\begin{array}{c}
\\
\infer[(Ax_{iw})]{\Gamma, x:\tA \vdashr x:\tA}{\w
\notin\mathcal{R}}\quad\quad \infer[(Ax_{ew})]{x:\tA \vdashr
x:\tA}{\w \in\mathcal{R}}
\\ \\
\infer[(\to_R)]{\Gamma\vdashr \lambda x.t:\tA\to
\tB}
                        {\Gamma, x:\tA \vdashr t:\tB} \quad\quad
\infer[(\to_L)]{\Gamma \cup_{\co} \Delta; \tA \to \tB \vdashr
t::k:\tC}
                        {\Gamma \vdashr t:\tA & \Delta;\tB \vdashr k:\tC}\\ \\
\infer[(Sel)]{\Gamma;\tA \vdashr \bindx t:\tB}
                    {\Gamma,x:\tA\ \vdashr t:\tB} \quad\quad
\infer[(Cut)]{\Gamma \cup_{\co} \Delta \vdashr tk:\tB}
                    {\Gamma \vdashr t:\tA & \Delta; \tA \vdashr k:\tB}\\ \\
\infer[(Cont_t)]{\Gamma, z:\tA \vdashr \cont{z}{x}{y}{t}:\tB}
                    {\Gamma, x:\tA,y:\tA \vdashr t:\tB & \co \in \mathcal{R}} \quad\quad
\infer[(Weak_t)]{\Gamma, x:\tA \vdashr \weak{x}{t}:\tB}
                    {\Gamma \vdashr t:\tB & \w \in \mathcal{R}}\\ \\
\infer[(Cont_k)]{\Gamma, z:\tA; \tB \vdashr \cont{z}{x}{y}{k}:\tC}
                    {\Gamma, x:\tA, y:\tA; \tB \vdashr k:\tC & \co \in \mathcal{R}} \quad\quad
\infer[(Weak_k)]{\Gamma, x:\tA; \tB \vdashr \weak{x}{k}:\tC}
                    {\Gamma; \tC \vdashr k:\tC & \w \in \mathcal{R}}\\ \\
\end{array}
$ }} \caption{$\lGR  \! \! \! \to$: Simply typed $\lGR$-calculi} \label{fig:simptyp-lGR}
\end{figure}

\pierre{
Simply typed $\lGR$-calculi implements intuitionistic sequent calculus cut elimination with
implicit/explicit structural rules through the \pierre{ {Curry-Howard correspondance.}  In  particular, $\lGo  \! \! \!  \to$ and
$\lGcw  \! \! \!  \to$ calculi correspond to the intuitionistic implicative
fragments of Kleene's systems $\mathrm{G}3$ and $\mathrm{G}1$ from
\cite{klee52}, respectively, except that the exchange
rule is made implicit here.} \jel{The exchange} rule could be made
explicit by considering the bases as lists instead of sets. The
system $\lGc \! \! \! \to$ corresponds to the intuitionistic sequent
calculus with explicit contraction and implicit weakening, \silv{whereas} 
the system $\lGw  \! \! \!  \to$ corresponds to the intuitionistic sequent
calculus with explicit weakening and implicit contraction.}

Modifications of $\lR  \! \! \!  \to$ and $\lGR  \! \! \!  \to$ systems could provide 
computational interpretations of substructural logics, different
from the usual approach via linear logic. For instance, if one
combines $(Ax_{ew})$ axiom and the other rules in $\w \notin
\mathcal{R}$ modality, the resulting system would correspond to
the logic without weakening i.e. to the variant of relevance
logic. Similarly, if we use $\cup_{\co}$ as disjoint union
together with the $\co \notin \mathcal{R}$ modality of the rest of
the system, correspondence with the variant of the logic without
contraction i.e. affine logic is obtained. The properties of
these systems will not be investigated in this paper.

Although the systems $\lR  \! \! \!  \to$ and $\lGR  \! \! \!  \to$ enjoy subject
reduction and strong normalisation, they (as expected) do not
assign types to all strongly normalising expressions, like
$\lambda x.xx$ and $\lambda x. x::x(\hat{y}.y)$.
 This is the
motivation for introducing intersection types in the next section.

\section{Intersection types for resource control cube}
\label{sec:int_types}

In this section we introduce intersection type assignment
which assign \emph{strict types} to $\lR$-terms and $\lGR$-expressions.
Intersection types for the $\lG$-calculus were introduced in~\cite{espighilivet07}.
Strict types were proposed in~\cite{bake92} and already used
in~\cite{espiivetlika08} and~\cite{ghilivetlesclika11} for characterising of strong
normalisation in $\lG$-calculus and in $\rcl$ and
$\llG$-calculi, respectively.

The syntax of types is defined as follows:
$$
\begin{array}{lccl}
\textrm{Strict Types}  & \tS & ::= & p \mid \tA \to \tS\\
\textrm{Types}      & \tA & ::= & \jel{\cap^n_i \tS_i}
\end{array}
$$
\noindent where $p$ ranges over a denumerable set of type atoms
and \jel{$\cap^n_i \tS_i = \tS_1 \cap...\cap \tS_n,\; n \geq 1$}. We
denote types by $\tA,\tB,\tC...$, \jel{strict types by
$\tS,\tT,\tR,\tU...$ and the set of all types by $\mathsf{Types}$}. We assume that the intersection operator is
idempotent, commutative and associative, \jel{ and that it has
priority over the arrow operator. Hence, we will omit parenthesis in
expressions like $(\cap^n_i \tT_i) \to \tS$.}


The definition of a basic type assignment, a basis and a basis
extension is analogous to the one given in
Section~\ref{sec:simple_types}.
\jel{The type assignments are of the form 
$$x_1:\tA_1,...,x_n:\tA_n \vdash M:\tS$$
 so that only strict types are assigned to terms.}

\begin{defi}
\rule{0in}{0in}
\begin{itemize}
\item[(i)] A union of bases with intersection types is defined in
the standard way:
$$
\begin{array}{lcl}
\Gamma \sqcup \Delta & = &  \{x:\tA\;|\;x:\tA \in \Gamma\; \& \; x \notin Dom(\Delta)\}\\
                    & \cup & \{x:\tA\;|\;x:\tA \in \Delta\; \& \; x \notin Dom(\Gamma)\} \\
                    & \cup & \{x:\tA \cap \tB\;|\;x:\tA \in \Gamma \; \& \; x:\tB \in \Delta\}.
\end{array}
$$
\item[(ii)] $\Gamma \unc \Delta$ represents $\Gamma \sqcup
\Delta$, if $\co \notin \mathcal{R}$, and the disjoint union $\Gamma,
\Delta$ otherwise.
\end{itemize}
\end{defi}




\subsection{Intersection types for $\lR$}

The type assignment systems $\lR \cap$ \silv{for the natural deduction ND-base base of the resource control cube} are given in
Figure~\ref{fig:inttyp-lR}. \jel{The rules that correspond to each of the four
$\lR \cap$-systems are given in Figure~\ref{fig:lR-cap}. 

All systems are syntax-directed i.e.
the intersection operator is incorporated into already existing rules of the simply-typed systems.
Intersection elimination is managed by the axioms $(Ax_{iw})$, $(Ax_{ew})$ and the contraction rule $(Cont)$,
whereas the intersection introduction is performed by the arrow elimination rule $(\to_E)$. Notice that in the $(\to_E)$ rule,
$Dom(\Delta_1)= \ldots = Dom(\Delta_n)$. The explicit contraction is naturally connected to intersection, because
if some data is used twice, once as data of type $\alpha$ and once as data of type $\beta$,
that data should be of type $\alpha \cap \beta$.}

\begin{figure}[hbtp]
\centerline{ \framebox{ $
\begin{array}{c}
\\
\infer[(Ax_{iw})]{\Gamma, x:\cap^n_i \tS_i\vdashr x:\tS_i}
                        {\w \notin \mathcal{R}} \quad\quad
\infer[(Ax_{ew})]{x:\cap^n_i \tS_i \vdashr x:\tS_i}
                        {\w \in \mathcal{R}}\\ \\
\infer[(\to_I)]{\Gamma\vdashr \lambda x.M:\tA\to \tS}
                        {\Gamma, x:\tA \vdashr M:\tS}\\ \\
\infer[(\to_E)]{\Gamma \unc (\Delta_1 \sqcup ... \sqcup \Delta_n)  \vdashr MN:\tS}
                    {\Gamma \vdashr M:\cap^n_i \tT_i \to \tS & \Delta_1 \vdashr N:\tT_1 & ... & \Delta_n \vdashr N:\tT_n}\\ \\
\infer[(Cont)]{\Gamma, z:\tA \cap \tB \vdashr
\cont{z}{x}{y}{M}:\tS}
                    {\Gamma, x:\tA,y:\tB \vdashr M:\tS & \co \in \mathcal{R}} \quad\quad
\infer[(Weak)]{\Gamma, x:\tA \vdashr \weak{x}{M}:\tS}
                    {\Gamma \vdashr M:\tS & \w \in \mathcal{R}}\\ \\
\end{array}
$ }} \caption{$\lR \cap$: $\lR$-calculi with intersection types}
\label{fig:inttyp-lR}
\end{figure}

\begin{figure}[hbtp]
\begin{tabular}{c|c}
$\lR \cap$-systems & type assignment rules \\ \hline $\lo \cap$ & $(Ax_{iw})$, $(\to_I)$, $(\to_E)$ \\
\hline  $\lc \cap$ & $(Ax_{iw})$, $(\to_I)$, $(\to_E)$, $(Cont)$ \\
\hline  $\lw \cap$ & $(Ax_{ew})$, $(\to_I)$, $(\to_E)$, $(Weak)$ \\
\hline  $\lcw \cap$ & $(Ax_{ew})$, $(\to_I)$, $(\to_E)$, $(Cont)$, $(Weak)$
\end{tabular}
\caption{Four ND intersection type systems} \label{fig:lR-cap}
\end{figure}


The proposed systems satisfy the following \silv{properties}. 

\begin{prop}[Generation lemma for $\lR\cap$]
\label{prop:intGL} \rule{0in}{0in}
\begin{itemize}
\item[(i)] \sil{For $\w \notin \mathcal{R}$: $\;\Gamma \vdashr
x:\tT\;\;$ iff there exist $\tS_{i}, i = 1,\ldots, n$ such that $\;x:\tT \cap (\cap_{i}^{n}\tS_{i}) \in \Gamma.$ }
\item[(ii)] \sil{For $\w
\in \mathcal{R}$: $\;\Gamma \vdashr x:\tT\;\;$ iff there exist $\tS_{i}, i = 1,\ldots, n$ such that$\;x:\tT \cap
(\cap_{i}^{n}\tS_{i}) = \Gamma.$}
\item[(iii)] $\Gamma \vdashr \lambda x.M:\tT\;\;$
iff there exist $`a$ and $`s$ such that $\;\tT\equiv
\tA\rightarrow \tS\;\;$ and $\;\Gamma,x:\tA \vdashr M:\tS.$
\item[(iv)] \sil{$\Gamma \vdashr MN:\tS\;\;$ iff}
\jel{$\;\Gamma=\Gamma' \unc \Delta,\;\Delta=\Delta_{1} \sqcup
\ldots \sqcup \Delta_{n}$} \sil{and there exist $\tT_i,\;i  = 1,
\ldots, n$ such that $\Gamma' \vdashr M:\cap_{i}^{n}\tT_i\to \tS$
and for all $i \in \{1, \ldots, n\}$, $\;\Delta_{i} \vdashr
N:\tT_i.$}
\item[(v)] $\Gamma \vdashr \weak{x}{M}:\tS\;\;$ iff
there exist $`G', `b$ such that $\;\Gamma=\Gamma', x:\tB$ and
$\;\Gamma' \vdashr M:\tS.$
\item[(vi)] $\Gamma \vdashr
\cont{z}{x}{y}{M}:\tS\;\;$ iff there exist $`G', `a, `b$ such that
$\;\Gamma=\Gamma', z:`a\cap`b$ and \\
$\;\Gamma', x:\tA, y:\tB \vdashr M:\tS.$
\end{itemize}
\end{prop}
\begin{proof} The proof is straightforward since all the rules are
syntax-directed. 
\end{proof}

\jel{
\begin{prop}[Substitution lemma for $\lR\cap$]
\label{prop:sub-lemma-nd}
\rule{0in}{0in}
If $\;\Gamma, x:\jel{\cap_i^{n}}\tT_i \vdashr M:\tS\;$
and for all $i \sil{\in \{1, \ldots, n\}}$, $\;\Delta_i \vdashr N:\tT_i$, then $\;\Gamma \unc
\jel{(\Delta_1 \sqcup ... \sqcup \Delta_n)} \vdashr
M\isub{N}{x}:\tS.$
\end{prop}

\begin{prop}[Subject equivalence for $\lR\cap$] \label{prop:subequiv}
\rule{0in}{0in}
For every $\lR$-term $M$: if $\;\Gamma \vdashr
M:\tS\;$ and $M \equiv M'$, then $\;\Gamma \vdashr M':\tS.$
\end{prop}

\begin{prop}[Subject reduction for $\lR\cap$]
\label{prop:sr}
\rule{0in}{0in}
For every $\lR$-term $M$: if $\;\Gamma \vdashr
M:\tS\;$ and $M \to M'$, then $\;\Gamma \vdashr M':\tS.$
\end{prop}
}


\subsection{Intersection types for $\lGR$}

The type assignment systems $\lGR \cap$ \silv{for the sequent LJ-base base of the resource control cube}  are given in
Figure~\ref{fig:inttyp-lGR}. \jel{The rules that correspond to each of the four
$\lGR \cap$-systems are given in Figure~\ref{fig:lGR-cap}. 

As in $\lR\cap$, no new rules are added compared to $\lGR \!\!\!\to$ in order to manage intersection. In the style of sequent calculus, left intersection introduction is managed by the axioms $(Ax_{iw})$, $(Ax_{ew})$ and the contraction rules $(Cont_t)$ and $(Cont_k)$, whereas the right intersection introduction is performed by the cut rule $(Cut)$ and left arrow introduction rule $(\to_L)$. In these two rules $Dom(\Gamma_1)= \ldots = Dom(\Gamma_n)$.}

In order to explain the rule $(\to_L)$ in more details let us consider a simple case $n=2$ and $m=2$.
\[
\infer[(\to_L)]{(\Gamma_1 \sqcup  \Gamma_2) \unc \Delta;
 (\tS_1 \cap \tS_2 \to \tT_1)\cap (\tS_1 \cap \tS_2 \to \tT_2)  \vdashr
            t::k:\tR}
             {\Gamma_1 \vdashr t:\tS_1  & \Gamma_2 \vdashr t:\tS_2 & \Delta; \tT_1 \cap \tT_2 \vdashr k:\tR}
\]

\silv{Although one would expect in the stoup
$\tS_1 \cap \tS_2 \to \tT_1 \cap \tT_2$, this is not a type according to the definition of intersection types at the beginning of this section. Therefore the corresponding type in the stoup is $(\tS_1 \cap \tS_2 \to \tT_1)\cap (\tS_1 \cap \tS_2 \to \tT_2)$. This difficulty does not exist in natural deduction, because natural deduction is isomorphic to a fragment of $\lG$, where the selection rule $(Sel)$ is restricted to $\bindx x$. In that fragment the $(\to_L)$ rule would be simpler with strict types in the stoup. Hence, the presented  $(Sel)$ and  $(\to_L)$ rules keep the full power of the sequent calculus.}
\begin{figure}[hbtp]
\centerline{ \framebox{
$
\begin{array}{c}
\\
\infer[(Ax_{iw})]{\Gamma, \sil{x:\cap_{i}^{n}} \tS_i \vdashr x:\tS_i}{\w
\notin\mathcal{R}}\quad\quad
\infer[(Ax_{ew})]{\sil{x:\cap_{i}^{n}} \tS_i
\vdashr x:\tS_i}{\w \in\mathcal{R}}
\\ \\
\infer[(\to_R)]{\Gamma\vdashr \lambda x.t:\tA\to
\tS}
                        {\Gamma, x:\tA \vdashr t:\tS} \quad\quad
\infer[(Sel)]{\Gamma; \tA \vdashr \bindx t:\tS}
                    {\Gamma,x:\tA\ \vdashr t:\tS}\\ \\
\infer[(\to_L)]{(\Gamma_1 \sqcup ... \sqcup \Gamma_n) \unc \Delta;
\sil{\cap_j^{m}(\cap_i^{n}\tS_i} \to \tT_j) \vdashr
            t::k:\tR}
             {\Gamma_1 \vdashr t:\tS_1 & ... & \Gamma_n \vdashr t:\tS_n & \Delta;\sil{\cap_j^{m}}\tT_j \vdashr k:\tR}\\ \\
\infer[(Cut)]{(\Gamma_1 \sqcup ... \sqcup \Gamma_n) \unc
                \Delta \vdashr tk:\tT}
                {\Gamma_1 \vdashr t:\tS_1 & ... & \Gamma_n \vdashr t:\tS_n & \Delta; \sil{\cap_{i}^{n}} \tS_i \vdashr k:\tT}\\ \\
\infer[(Cont_t)]{\Gamma, z:\tA \cap \tB \vdashr
\cont{z}{x}{y}{t}:\tS}
                    {\Gamma, x:\tA,y:\tB \vdashr t:\tS & \co \in \mathcal{R}} \quad\quad
\infer[(Weak_t)]{\Gamma, x:\tA \vdashr \weak{x}{t}:\tS}
                    {\Gamma \vdashr t:\tS & \w \in \mathcal{R}}\\ \\
\infer[(Cont_k)]{\Gamma, z:\tA \cap \tB; \tC \vdashr
\cont{z}{x}{y}{k}:\tS}
                    {\Gamma, x:\tA, y:\tB; \tC \vdashr k:\tS & \co \in \mathcal{R}} \quad\quad
\infer[(Weak_k)]{\Gamma, x:\tA; \tC \vdashr \weak{x}{k}:\tS}
                    {\Gamma; \tC \vdashr k:\tS & \w \in \mathcal{R}}\\ \\
\end{array}
$ }} \caption{$\lGR \cap$: $\lGR$-calculi with intersection
types} \label{fig:inttyp-lGR}
\end{figure}

\begin{figure}[hbtp]
\begin{tabular}{c|c}
$\lR \cap$-systems & type assignment rules \\ \hline $\lo \cap$ & $(Ax_{iw})$, $(\to_R)$, $(\to_L)$, $(Sel)$, $(Cut)$ \\
\hline  $\lc \cap$ & $(Ax_{iw})$, $(\to_R)$, $(\to_L)$, $(Sel)$, $(Cut)$, $(Cont_t)$, $(Cont_k)$ \\
\hline  $\lw \cap$ & $(Ax_{ew})$, $(\to_R)$, $(\to_L)$, $(Sel)$, $(Cut)$, $(Weak_t)$, $(Weak_k)$ \\
\hline  $\lcw \cap$ & $(Ax_{ew})$, $(\to_R)$, $(\to_L)$, $(Sel)$, $(Cut)$, $(Cont_t)$, $(Cont_k)$, $(Weak_t)$, $(Weak_k)$
\end{tabular}
\caption{Four LJ intersection type systems} \label{fig:lGR-cap}
\end{figure}

The proposed systems satisfy the following properties.

\begin{prop}[Generation lemma for $\lGR\cap$]
\label{prop:intGL}
\rule{0in}{0in}
\begin{itemize}
\item[(i)] \sil{For $\w \notin \mathcal{R}$: $\;\Gamma \vdashr
x:\tT\;\;$  iff there exist $\tS_{i}, i = 1,\ldots, n$ such that $\;x:\tT \cap (\cap_{i}^{n}\tS_{i}) \in \Gamma.$ }
\item[(ii)] \sil{For $\w
\in \mathcal{R}$: $\;\Gamma \vdashr x:\tT\;\;$  iff there exist $\tS_{i}, i = 1,\ldots, n$ such that $\;x:\tT \cap (\cap_{i}^{n}\tS_{i}) = \Gamma.$}
\item[(iii)] $\Gamma \vdashr \lambda x.t:\tT\;\;$
iff there exist $`a$ and $`s$ such that $\;\tT\equiv
\tA\rightarrow \tS\;\;$ and $\;\Gamma,x:\tA \vdashr t:\tS.$
\item[(iv)] $\Gamma;\tA \vdashr \bindx t:\tS\;\;$ iff
$\;\Gamma,x:\tA \vdashr t:\tS.$
\item[(v)] \sil{$\Gamma \vdashr tk:\tS\;\;$ iff}\jel{
$\;\Gamma=\Gamma' \unc \Delta,\;\Gamma'=\Gamma'_1 \sqcup ...
\sqcup \Gamma'_n$} \sil{and there exist $\tT_i, i = 1, \ldots, n$
such that for all $i \in \{1, \ldots, n\},\;$} \silv{the following holds} \sil{ $\Gamma'_i \vdashr
t:\tT_i$, and $\;\Delta;\cap_{1}^{n} \tT_i \vdashr k:\tS.$}
\item[(vi)] \sil{$\Gamma;\tC \vdashr t::k:\tR\;\;$ iff}\jel{
$\;\Gamma=\Gamma' \unc \Delta,\;\Gamma'=\Gamma'_1 \sqcup ...
\sqcup \Gamma'_n$}, \sil{$\tC \equiv \cap_j^{m} (\cap_i^{n}\tS_i
\to \tT_j)\;$ and for all $i \in \{1, \ldots, n\},$} \silv{the following holds} \sil{$\;\Gamma'_i
\vdashr t:\tS_i$ and $\;\Delta;\cap_j^{m}\tT_j \vdash k:\tR\;$. }
\item[(vii)] $\Gamma
\vdashr \weak{x}{t}:\tS\;\;$ iff there exist $`G', `b$ such that
$\;\Gamma=\Gamma', x:\tB$ and $\;\Gamma' \vdashr t:\tS.$
\item[(viii)] $\Gamma; \tC \vdash \weak{x}{k}:\tS\;\;$ iff there
exist $`G,`b$ such that $\;\Gamma=\Gamma', x:\tB$ and $\;\Gamma;
\tC \vdashr k:\tS.$
\item[(ix)] $\Gamma \vdashr
\cont{z}{x}{y}{t}:\tS\;\;$ iff there exist $`G', `a, `b$ such that
$\;\Gamma=\Gamma', z:`a\cap`b$ and \\
$\;\Gamma', x:\tA, y:\tB \vdashr t:\tS.$
\item[(x)] $\Gamma;
\tE \vdashr \cont{z}{x}{y}{k}:\tS\;\;$
iff there exist $`G',`a, `b$ such that $\;\Gamma=\Gamma', z: \tA \cap \tB$ and\\
$\;\Gamma, x:\tA, y:\tB; \tE \vdashr k:\tS.$
\end{itemize}
\end{prop}
\begin{proof} The proof is straightforward since all the rules are
syntax-directed. 
\end{proof}

\begin{prop}
\label{prop:pres-of-FV} If $e \to e'$, then $Fv(e) \supseteq Fv(e')$ and if $\w`:\mathsf{R}$, $Fv(e)=Fv(e')$.
\end{prop}


\pierre{\begin{prop}
\label{prop:bases-weak}
\rule{0in}{0in}
\begin{itemize}
\item[(i)]  If $\;\Gamma \vdashr t:\tS\;$, then
$\;Dom(\Gamma)\supseteq Fv(t)$ and if $\w \in \mathcal{R}$, $\;Dom(\Gamma)= Fv(t).$ 
\item[(ii)] If $\;\Gamma; \tA \vdashr k:\tS\;$, then
$\;Dom(\Gamma)\supseteq Fv(k)$ and  if $\w \in \mathcal{R}$ $\;Dom(\Gamma)= Fv(k)$.
\end{itemize}
\end{prop}
}
\begin{prop}[Substitution lemma \jel{for $\lGR\cap$}]
\label{prop:sub-lemma}
\rule{0in}{0in}
\begin{itemize}
\item[(i)] If $\;\Gamma, x:\jel{\cap_i^{n}}\tT_i \vdashr t:\tS\;$
and for all $i$, $\;\Delta_i \vdashr u:\tT_i$, then $\;\Gamma \unc
\jel{(\Delta_1 \sqcup ... \sqcup \Delta_n)} \vdashr
t\isub{u}{x}:\tS.$
\item[(ii)] If $\;\Gamma, x:\jel{\cap_i^{n}} \tT_i;\tA \vdashr
k:\tS\;$ and for all $i$, $\;\Delta_i \vdashr u:\tT_i$, then
$\;\Gamma \unc \jel{(\Delta_1 \sqcup ... \sqcup \Delta_n)};\tA
\vdashr k\isub{u}{x}:\tS.$
\end{itemize}
\end{prop}

\begin{prop}[Append lemma]
\label{prop:app-lemma} \rule{0in}{0in} If $\;\Gamma;\tA \vdashr
k:\tT_i\;$ for all $i$, and $\;\Delta;\jel{\cap_i^{n}} \tT_i
\vdashr k':\tS$, then $\;\Gamma \unc \Delta;\tA \vdashr
\app{k}{k'}:\tS.$
\end{prop}

\jel{\begin{prop}[Subject equivalence \jel{for $\lGR\cap$}] \label{prop:subequiv-sequent}
\rule{0in}{0in}
\begin{itemize}
\item[(i)] For every $\lGR$-term $t$: if $\;\Gamma \vdashr
t:\tS\;$ and $t \equiv t'$, then $\;\Gamma \vdashr t':\tS.$
\item[(ii)] For every $\lGR$-context $k$: if $\;\Gamma; \tA
\vdashr k:\tS\;$ and $k \equiv k'$, then $\;\Gamma; \tA \vdashr
k':\tS.$
\end{itemize}
\end{prop}
}
\begin{prop}[Subject reduction \jel{for $\lGR\cap$}]
\label{prop:sr-sequent}
\rule{0in}{0in}
\begin{itemize}
\item[(i)] For every $\lGR$-term $t$: if $\;\Gamma \vdashr
t:\tS\;$ and $t \to t'$, then $\;\Gamma \vdashr t':\tS.$
\item[(ii)] For every $\lGR$-context $k$: if $\;\Gamma; \tA
\vdashr k:\tS\;$ and $k \to k'$, then $\;\Gamma; \tA \vdashr
k':\tS.$
\end{itemize}
\end{prop}
\begin{proof} \pierre{The property is stable by context. Therefore we can assume that the
  reduction takes place at the outermost position of the term $t$ (resp. of the context $k$).}  Here we just
show several cases.
\sil{We will use GL as an abbreviation for Generation lemma for $\lGR\cap$ (Lemma~\ref{prop:intGL}).}\\
Case $(\beta)$: Let $\Gamma
\vdashr (\lambda x.t)(u::k):\tS$. We are showing that
  $\Gamma' \vdashr u (\bindx tk):\tS$.\\
 From $\Gamma \vdashr (\lambda x.t)(u::k):\tS\;$ and from GL(v) it follows that
 $\Gamma = \Gamma'' \unc \Delta$, \jel{$\Gamma''=\Gamma''_1 \sqcup ... \sqcup \Gamma''_m$} and that there is a type $\sil{\cap_j^{m}} \tT_j$ such that
 \sil{for all $j=1, \ldots, m$,} $\Gamma''_j \vdashr \lambda x.t:\tT_j,\;$ and
 $\Delta;\sil{\cap_j^{m}}\tT_j \vdashr u::k:\tS$.
From GL(iii) we have that \sil{for all $j=1, \ldots, m$,} $\tT_j
\equiv \tA_j \to \tR_j$ and $\Gamma''_j, x:\tA_j \vdashr t:\tR_j$.
From GL(vi) it follows that
 $\Delta = \Delta' \unc \Delta''$, \jel{$\Delta'=\Delta'_1 \sqcup ...
 \sqcup \Delta'_n$},
$\jel{\Delta'_i} \vdashr u:\tS_i\;$ and $\Delta'';
\sil{\cap_j^{m}}\tR_j
 \vdashr k:\tS$, for $\sil{\cap_j^{m}} \tT_j \equiv \sil{\cap_j^{m}(\cap_i^{n}}\tS_i \to \tR_j)$. Also, we conclude that $\tA_j \equiv \sil{\cap_i^{n}}\tS_i$. Now,
 \[
 \prooftree
      \Delta'_1 \vdashr u:\tS_1\;...\;\Delta'_n \vdashr u:\tS_n
      \prooftree
          \prooftree
              \Gamma''_1, x:\sil{\cap_i^{n}}\tS_i \vdashr t:\tR_1\;...\;\Gamma''_m, x:\sil{\cap_i^{n}}\tS_i \vdashr t:\tR_m\;\;\;\;\Delta''; \sil{\cap_j^{m}}\tR_j
 \vdash k:\tS
              \justifies
              \Gamma'' \unc \Delta'',x:\sil{\cap_i^{n}}\tS_i \vdashr tk:\tS
              \using{(Cut)}
          \endprooftree
          \justifies
          \Gamma''\unc \Delta'';\sil{\cap_i^{n}}\tS_i \vdashr \bindx tk:\tS
          \using{(Sel)}
      \endprooftree
      \justifies
      \jel{\Gamma''\unc \Delta'' \unc (\Delta'_1 \sqcup ... \sqcup \Delta'_n)} \vdashr u (\bindx tk):\tS.
      \using{(Cut)}
 \endprooftree
 \]
\jel{which is exactly what we wanted, since $\Gamma = \Gamma''\unc \Delta'' \unc (\Delta'_1 \sqcup ... \sqcup \Delta'_n)$.}\\
Case $(\gamma_6)$: Let $\Gamma,x:\tA;\tB \vdashr
\cont{x}{x_1}{x_2}{(t::k)}:\tS$.
We are showing that $\Gamma,x:\tA;\tB \vdash t::\cont{x}{x_1}{x_2}{k}:\tS$.\\
From $\Gamma,x:\tA;\tB \vdashr \cont{x}{x_1}{x_2}{(t::k)}:\tS$ by
GL(x) we have that $\tA \equiv \tC \cap \tD$ and
$\Gamma,x_1:\tC,x_2:\tD;\tB \vdashr (t::k):\tS$. Next, GL(vi)
and the reduction side-condition $x_1,x_2 \notin Fv(t)$ imply that
$\Gamma = \jel{\Gamma'_1 \sqcup ... \sqcup \Gamma'_n}, \Gamma''$, $\tB \equiv \sil{\cap_j^{m}(\cap_i^{n}}\tT_i \to
\tR_j)$, $\jel{\Gamma'_i}\vdashr t:\tT_i$ \jel{for $i=1...n$} and
$\Gamma'',x_1:\tC,x_2:\tD;\sil{\cap_j^{m}}\tR_j \vdashr k:\tS$. Now,
\[
 \prooftree
      \jel{\Gamma'_1\vdashr t:\tT_1\;...\;\Gamma'_n\vdashr t:\tT_n}
      \prooftree
          \Gamma'',x_1:\tC,x_2:\tD;\sil{\cap_j^{m}}\tR_j \vdashr k:\tS
          \justifies
          \Gamma'',x:\tC \cap \tD; \sil{\cap_j^{m}}\tR_j \vdashr
          \cont{x}{x_1}{x_2}{k}:\tS
          \using{(Cont_k)}
      \endprooftree
      \justifies
      \jel{\Gamma'_1 \sqcup ... \sqcup \Gamma'_n},\Gamma'',x:\tC \cap \tD;\sil{\cap_j^{m}(\cap_i^{n}}\tT_i \to
\tR_j)  \vdashr t::\cont{x}{x_1}{x_2}{k}:\tS
      \using{(\to_L)}
 \endprooftree
 \]
which is exactly what we needed.\\
Case $(\omega_4)$: Let $\Gamma,y:\tA;\tB \vdashr \bindx
\weak{y}{t}:\tS$.
We are showing that $\Gamma,y:\tA;\tB \vdashr \weak{y}{\bindx t}:\tS$.\\
From $\Gamma,y:\tA;\tB \vdashr \bindx \weak{y}{t}:\tS$ by
GL(iv) we have that $\Gamma,y:\tA,x:\tB \vdashr
\weak{y}{t}:\tS$ and afterwards by GL(viii) that $\Gamma,x:\tB
\vdashr t:\tS$. Now,
\[
 \prooftree
      \prooftree
          \Gamma,x:\tB \vdashr t:\tS
          \justifies
          \Gamma;\tB \vdashr \bindx t:\tS
          \using{(Sel)}
      \endprooftree
      \justifies
      \Gamma,y:\tA;\tB \vdashr \weak{y}{\bindx t}:\tS.
      \using{(Weak_k)}
 \endprooftree
 \]
 
\end{proof}


\section{Typeability $\Rightarrow$ SN in all systems of the resource control cube}
\label{sec:typeSN}


\subsection{Typeability $\Rightarrow$ SN in $\lR \cap$}
\label{sec:reducibility}

The \emph{reducibility method} is a well known approach for proving reduction properties
of $\lambda$-terms typeable in different type assignment systems.
It was introduced by Tait~\cite{tait67} for proving the strong normalisation property
of simply typed $\lambda$-calculus.
It was developed further to prove \emph{strong normalisation property} of various calculi
in~\cite{tait75,gira71,kriv90,ghil96}, \emph{confluence} 
of $\beta \eta$-reduction in~\cite{kole85,statI85} 
and to characterise certain classes of $\lambda$-terms
such as strongly normalising, normalising, head normalising,
and weak head normalising terms by their typeability in various intersection type systems
in~\cite{gall98,dezahonsmoto00,dezaghil02,dezaghillika04}.

The principal idea of the reducibility method is to connect the
terms typeable in a certain type assignment system with the terms satisfying certain reduction properties
(e.g., strong normalisation, confluence).
To this aim, types are interpreted by suitable sets of lambda terms which satisfy certain
realizability properties.
Then the soundness of type assignment with respect to these interpretations is obtained.
As a consequence of soundness, every typeable term belongs to the interpretation of its type, and as such satisfies a desired reduction property.

In the remainder of the paper we consider $\LR$ to be the applicative structures whose domains are $\lR$-terms
and where the application is just the application of $\lR$-terms. In addition, $\LRc$ is the set of \emph{closed} $\lR$-terms. We first recall some notions from~\cite{bare92}.

\begin{defi}
For $\vM, \vN \subseteq \LRc$, we define $\vM \fsto \vN \subseteq \LRc$ as
$$\vM \fsto \vN =  \{M  \in \LRc \mid \forall N \in \vM \quad MN \in \vN \}.$$
\end{defi}

First of all, we introduce the following notion of {\em type interpretation\/}.

\begin{defi}[Type interpretation]
\label{def:typeInt}
The type interpretation  $\ti{-} : \tlam \to 2^{\LRc}$ is defined by:
  \begin{itemize}
    \item[($I 1$)] $\ti{p} = \SNc$, where $p$ is a type atom;
    \item[($I 2$)] $\ti{\cap_{i}^{n} \tS_{i}}  = \cap_{i}^{n} \ti{\tS_{i}}$;
    \item[($I 3$)] $\ti{\tA \to \tS} = \ti{\tA} \fsto \ti{\tS}$.
\end{itemize}
\end{defi}

Next, we introduce the notion of \emph{saturation property}, obtained by modifying the saturation property given in~\cite{bare92}.


\begin{defi}\label{def:var+sat}
\rule{0in}{0in}
A set $X \subseteq \SNc$ 
\silv{is called \emph{$\mathcal{R}$-saturated}
if it satisfies the following two properties:}
    \begin{itemize}
    \item $\INH(\vX)$:
         $(\exists n \geq 0) \; (\forall p \geq n) \;
                        \lambda x_1 \ldots \lambda x_p.\lambda y.y \in \vX.$
    \item $\SAT_{\beta}(\vX)$:
         $(\forall n \geq 0) \, (\forall M_1, \ldots, M_n \in \SNc)\,
    (\forall N \in \SNc)$ \\
         $ \hspace*{10mm}  M\isub{N}{x} M_1 \ldots M_n \in \vX \;  \Rightarrow \;  (\lambda x.M)N M_1 \ldots M_n \in \vX. $
    \end{itemize}
\end{defi}

%
%

\begin{prop} \label{prop:saturSets}
Let $\vM, \vN \subseteq \LRc$.
\begin{itemize}
    \item[(i)] $\SNc$ is $\mathcal{R}$-saturated.
    \item[(ii)] If $\vM$ and $\vN$ are $\mathcal{R}$-saturated, then $\vM \fsto \vN$ is $\mathcal{R}$-saturated.
    \item[(iii)] If $\vM$ and $\vN$ are $\mathcal{R}$-saturated, then $\vM \cap \vN$ is $\mathcal{R}$-saturated.
    \item[(iv)] For all types $\varphi \in \mathsf{Types}$, $\ti{\varphi}$ is $\mathcal{R}$-saturated.
\end{itemize}
\end{prop}

\begin{proof}
\rule{0in}{0in}
\begin{itemize}
\item[(i)]
$\SNc \subseteq \SNc$ and the condition $\INH(\SNc)$ trivially hold.\\
For $\SAT_{\beta}(\SNc)$,
\sil{suppose that
	$$M\isub{N}{x} M_1 \ldots M_n \in \SNc \mbox{ and } N, M_1, \ldots, M_n \in \SNc.$$
We have to prove that
	$$(\lambda x.M)N M_1 \ldots M_n \in \SNc. $$
Since $M\isub{N}{x}$ is a subterm of a term in $\SNc$, we know that $M \in \SNc$. By assumption, $N, M_{1}, \ldots, M_{n} \in \SNc$, so the reductions inside of these terms terminate. After finitely many reduction steps, we obtain
	$$(\lambda x.M)N M_1 \ldots M_n \rightarrow \ldots \rightarrow (\lambda x.M')N' M'_1 \ldots M'_n$$
\noindent where $M \rightarrow M',\; \; N \rightarrow N', \;\; M_{1} \rightarrow M'_{1}, \ldots, M_{n} \rightarrow M'_{n}.$ After contracting $(\lambda x.M')N' M'_1 \ldots M'_n$ to $M'\isub{N'}{x} M'_1 \ldots M'_n$, we obtain a reduct of $M\isub{N}{x} M_1 \ldots M_n \in \SNc$. Hence, also the initial term $(\lambda x.M)N M_1 \ldots M_n \in \SNc. $}
\item[(ii)]
First, we prove that $\vM \fsto \vN \subseteq \SNc$.
Suppose that $M \in \vM \fsto \vN$. 
Then, for all $N \in \vM,\; MN \in \vN$. Since $\vM$ is $\mathcal{R}$-saturated, $\INH(\vM)$ holds and we have that
$$(\exists n \geq 0) \; (\forall p \geq n) \;  \lambda x_1 \ldots \lambda x_p.\lambda y.y \in \vM, \mbox{ hence }$$
$$(\exists n \geq 0) \; (\forall p \geq n) \;  M(\lambda x_1 \ldots \lambda x_p.\lambda y.y) \in \vN \subseteq \SNc.$$
From here we can deduce that $M \in \SNc$.\\ 

Next, we prove that $\INH(\vM \fsto \vN)$ holds \silv{i.e.  
\[(\exists n \geq 0) \; (\forall p \geq n) \; \lambda x_1 \ldots \lambda x_p.\lambda y.y \in \vM \fsto \vN.\]
By the induction hypothesis $\vN$ is $\mathcal{R}$-saturated, hence $\INH(\vN)$ and $\SAT(\vN)$ hold.}
Since $\INH(\vN)$ holds, we have that
$$(\exists n \geq 0) \; (\forall p \geq n) \;  \lambda x_1 \ldots \lambda x_p.\lambda y.y \in \vN.$$ By taking $p=n+1$ we obtain that for all $N \in \vM$,
$$(\lambda x_1 \ldots \lambda x_n \lambda x_{n+1}.\lambda y.y)N  \rightarrow \lambda x_1 \ldots \lambda x_n \lambda y.y$$ 
\silv{By $\INH(\vN)$ we get that $\lambda x_1 \ldots \lambda x_n \lambda y.y \in \vN$ and then by $\SAT(N)$ we get that $(\lambda x_1 \ldots \lambda x_n \lambda x_{n+1}.\lambda y.y)N \in \vN$.
\[(\lambda x_1 \ldots \lambda x_n \lambda x_{n+1}.\lambda y.y) \in \vM \fsto \vN.
\]
Similarly, we can prove the above property for all $p \geq n+1$, i.e.
\[
(\lambda x_1 \ldots \lambda x_p.\lambda y.y) \in \vM \fsto \vN.
\]
}

Finally, for $\SAT_{\beta}(\vM \fsto \vN)$, 
\sil{suppose that
	$$M\isub{N}{x} M_1 \ldots M_n \in \vM \fsto \vN \mbox{ and } N, M_1, \ldots, M_n \in \SNc.$$
This means that for all $P \in \vM$
	$$M\isub{N}{x} M_1 \ldots M_nP \in \vN.$$
But $\vN$ is $\mathcal{R}$-saturated, so $\SAT_{\beta}(\vN)$ holds and we have that for all $P \in \vN$
	$$(\lambda x.M)N M_1 \ldots M_nP \in \vN$$
This means that
	$(\lambda x.M)N M_1 \ldots M_n \in \vM \fsto \vN. $}
\item[(iii)]
$\vM \cap \vN \subseteq \SNc$ is straightforward.\\
For $\INH(\vM \cap \vN)$, since $\INH(\vM)$ and $\INH(\vN)$ hold  we have that
$$(\exists n_{1} \geq 0) \; (\forall p_{1} \geq n_{1}) \;  \lambda x_1 \ldots \lambda x_{p_{1}}.\lambda y.y \in \vM$$
$$(\exists n_{2} \geq 0) \; (\forall p_{2} \geq n_{2}) \;  \lambda x_1 \ldots \lambda x_{p_{2}}.\lambda y.y \in \vN.$$
By taking $n=max(n_{1},n_{2})$ we obtain
$$(\exists n \geq 0) \; (\forall p \geq n) \;  \lambda x_1 \ldots \lambda x_p.\lambda y.y \in \vM \cap \vN, \mbox{ i.e. }\INH(\vM \cap \vN) \mbox{ holds.}$$
$\SAT_{\beta}(\vM \cap \vN)$ is straightforward.

\item[(iv)]
By induction on the construction of $\varphi \in \mathsf{Types}$.
\begin{itemize}
    \item If $\varphi \equiv p$, $p$ is type atom, then $\ti{\varphi} = \SNc$, so it is $\mathcal{R}$-saturated using (i).
    \item If $\varphi \equiv \tA \to \tS$, then $\ti{\varphi} = \ti{\tA} \fsto \ti{\tS}$. Since $\ti{\tA}$ and $\ti{\tS}$ are $\mathcal{R}$-saturated by IH,  we can use (ii).
    \item If $\varphi \equiv \cap_{i}^{n} \tS_{i}$, then $\ti{\varphi} =\ti{\cap_{i}^{n} \tS_{i}}  = \cap_{i}^{n} \ti{\tS_{i}}$ and for all $i=1, \ldots, n, \ti{\tS_{i}}$ are $\mathcal{R}$-saturated by IH, so we can use (iii).
\end{itemize}
\end{itemize}
\end{proof}

We further define a {\em valuation of terms\/} $\tei{-}_{\rho}: \LR \to \LRc$ and the {\em semantic satisfiability relation\/} $\modelsr$ which connects the type interpretation with the term valuation.

\begin{defi}  \label{def:val}

Let $\rho : {\tt var} \to \LRc$ be a valuation of term variables in $\LRc$.
Then, for the term $M \in \LR$ with free variables $Fv(M) = \{x_{1}, \ldots, x_{n}\}$,
the \emph{term valuation} $\tei{-}_\rho : \LR \to \LRc$ is defined as follows
$$\tei{M}_{\rho} = M[\isubs{\rho(x_{1})}{x_{1}}, \ldots, \isubs{\rho(x_{n})}{x_{n}}].$$

\end{defi}

\begin{lem}
\label{lem:sub}
\rule{0in}{0in}
    \begin{itemize}
    \item[(i)]
    $\tei{MN}_{\rho} \equiv \tei{M}_{\rho} \tei{N}_\rho.$
    \item[(ii)]
$\tei{\lambda x. M}_{\rho} N    \silv{\to} \tei{ M}_{\rho(\isubs{N}{x})} $, where $N \in \LRc$.
    \item[(iii)]
    $\tei{\weak{x}{M}}_{\rho} \equiv \tei{M}_{\rho}.$
    \item[(iv)]
    $\tei{\cont{z}{x}{y}{M}}_{\rho} \equiv
    \tei{M}_{\rho(\isubs{N}{x},\isubs{N}{y})}$
    where $N=\rho(z) \in \LRc$.
    \end{itemize}
\end{lem}

\begin{proof}
\rule{0in}{0in}
\begin{itemize}
	\item[(i)] Straightforward from the definition of substitution given in Figure~\ref{fig:sub-lR}.
	\item[(ii)]
	If  $Fv(\lambda x.M) = \{x_{1}, \ldots, x_{n}\}$, then\\
	$\tei{\lambda x.M}_{\rho} N \equiv
	  (\lambda x.M)[\isubs{\rho(x_{1})}{x_{1}}, \ldots, \isubs{\rho(x_{n})}{x_{n}}] N    \silv{\to}
        (M[\isubs{\rho(x_{1})}{x_{1}}, \ldots, \isubs{\rho(x_{n})}{x_{n}}])[\isubs{N}{x}]
    \equiv \\
	M[ \isubs{\rho(x_{1})}{x_{1}}, \ldots, \isubs{\rho(x_{n})}{x_{n}}, \isubs{N}{x}]$.\\
	The last equivalence holds, since all $\rho(x_{i})$ are closed terms so $x \not \in FV(\rho(x_{i}))$.
	\item[(iii)]
	If  $Fv(M) = \{x_{1}, \ldots, x_{n}\}$, then\\
	$\tei{\weak{x}{M}}_{\rho} \equiv
	  (\weak{x}{M})[\isubs{\rho(x)}{x}, \isubs{\rho(x_{1})}{x_{1}}, \ldots, \isubs{\rho(x_{n})}{x_{n}}] \equiv\\
	  \weak{Fv(\rho(x))}{M[\isubs{\rho(x_{1})}{x_{1}}, \ldots, \isubs{\rho(x_{n})}{x_{n}}]} \equiv \tei{M}_{\rho},$\\
	  since $Fv(\rho(x))=\emptyset$ because $\rho(x) \in \LRc$.
	\item[(iv)]
	If  $Fv(M) = \{x_{1}, \ldots, x_{n}\}$ and we denote $\rho(z)=N \in \LRc$, then\\
	$\tei{\cont{z}{x}{y}{M}}_{\rho} \equiv
	(\cont{z}{x}{y}{M})[\isubs{\rho(z)}{z}, \isubs{\rho(x_{1})}{x_{1}}, \ldots, \isubs{\rho(x_{n})}{x_{n}}] \equiv\\
	\cont{Fv(N)}{Fv(N_{1})}{Fv(N_{2})}{M}[\isubs{N_{1}}{x}, \isubs{N_{2}}{y}, \isubs{\rho(x_{1})}{x_{1}}, \ldots, 			 \isubs{\rho(x_{n})}{x_{n}}] \equiv\\
	M[\isubs{N}{x}, \isubs{N}{y}, \isubs{\rho(x_{1})}{x_{1}}, \ldots, \isubs{\rho(x_{n})}{x_{n}}] \equiv
	\tei{M}_{\rho(\isubs{N}{x},\isubs{N}{y})}$\\
	since $N \in \LRc$, hence $Fv(N)=\emptyset$, $N_{1}=N_{2}=N$ (no renaming takes place) and
	$Fv(N_{1})=Fv(N_{2})=\emptyset$.
\end{itemize}
\end{proof}

\begin{defi} \label{def:model}
\rule{0in}{0in}
\begin{itemize}
    \item [(i)] $\rho \modelsr M : \tA \quad \iff\ \quad \tei{M}_\rho \in \ti{\tA}$;
    \item [(ii)] $\rho \modelsr \Gamma \quad \iff\ \quad (\forall (x:\tA) \in \Gamma) \quad \rho(x)\in \ti{\tA}$;
    \item [(iii)] $\Gamma \modelsr M : \tA \quad \iff\ \quad (\forall \rho) \quad (\rho \modelsr \Gamma \Rightarrow \rho \modelsr M : \tA)$.
  \end{itemize}
\end{defi}


\begin{prop} [Soundness of $\lR \cap$] \label{prop:sound}
If $\Gamma \vdashr M:\tA$, then $\Gamma \modelsr M:\tA$.
\end{prop}

\begin{proof}
By induction on the derivation of $\Gamma \vdashr M:\tA$.
\begin{itemize}
\item
The last rule applied is $(Ax_{iw})$.
Suppose that $\rho \models \Gamma, x:\cap_{1}^{n} \tS_{i}$.
From here we deduce that $\rho(x) \in \ti{\cap_{1}^{n} \tS_{i}} = \cap_{1}^{n} \ti{\tS_{i}}$
which means that for all $i \in \{1, \ldots, n\}, \rho(x) \in \ti{\tS_{i}}$.
\item
$(Ax_{ew})$. Similar to the previous case.
\item
The last rule applied is $(\to_{I})$, i.e.,\
    $$\Gamma, x:\tA \vdashr M:\tS \Rightarrow
      \Gamma\vdashr \lambda x.M:\tA \to \tS.$$
By the IH $\Gamma, x:\tA \modelsr M: \tS$.
Suppose that $\rho \modelsr \Gamma$ and we want to show that $\rho \modelsr \lambda x.M: \tA \to \tS$.
We have to show that
$$\tei{\lambda x.M}_{\rho} \in \ti{\tA \to \tS} = \ti{\tA} \fsto \ti{\tS} \mbox{ i.e.}$$
$$\forall N \in \ti{\tA}. \; \tei{\lambda x.M}_{\rho}N \in \ti{\tS}.$$
Suppose that $N \in \ti{\tA}$.
Then, let us consider a new valuation $\rho'=\rho(\isubs{N}{x})$, which can be constructed because $N$ is a closed term. We have that $\rho' \modelsr \Gamma, x:\alpha$ since $\rho \modelsr \Gamma$, $x \not\in \Gamma$ and $\rho'(x)=N \in \ti{\tA}$. \silv{Then by the IH $\rho' \modelsr M:\sigma$, hence} we can conclude that $\tei{M}_{\rho'} \in \ti{\tS}$. Applying Lemma~\ref{lem:sub}(ii) we get $\tei{\lambda x.M}_{\rho}N    \silv{\to} \tei{M}_{\rho'}$.
Since $\tei{M}_{\rho'} \in \ti{\tS}$ and $\ti{\tS}$ is saturated, 
we obtain $\tei{\lambda x.M}_{\rho}N \in \ti{\tS}$.
\item
The last rule applied is $(\to_{E})$, i.e.,\
    $$ \hspace*{12mm} \Gamma \vdashr M: \cap_{1}^{n} \tT_i \to \tS,\;\;
    \Delta_{1} \vdashr N:\tT_1 \;\; \ldots \;\;
    \Delta_{n} \vdashr N:\tT_n \Rightarrow
    \Gamma \unc (\Delta_{1} \sqcup \ldots \sqcup \Delta_{n}) \vdashr MN:\tS.$$
By the IH $\Gamma \modelsr M:\cap_{1}^{n} \tT_i \to \tS$,
$\Delta_{1} \modelsr N : \tT_{1}, \ldots, \Delta_{n} \modelsr N : \tT_{n}$.
Suppose that $\rho \modelsr \Gamma \unc (\Delta_{1} \sqcup \ldots \sqcup \Delta_{n})$, \silv{in order to show $\tei{MN}_{\rho} \in \ti{\tS}$.}
Then  for all $x:\tA \in \Gamma \unc (\Delta_{1} \sqcup \ldots \sqcup \Delta_{n}), \; \rho(x) \in \tei{\tA}$.
This means that\\
\hspace*{3mm} if $x:\tA \in \Gamma$ and $x \not\in Dom(\Delta_{1}) \cap \ldots \cap Dom(\Delta_{n})$, then  $\rho(x) \in \tei{\tA}$;\\
\hspace*{3mm} if there is $i$, such that $x:\tA_{i} \in \Delta_{i}$ and $x \not\in Dom(\Gamma) \cap Dom(\Delta_{j}), j\not= i$, then $\rho(x) \in \tei{\tA_{i}}$;\\
\hspace*{3mm} if
$x: \tA \cap \tA_{1} \cap \ldots \cap \tA_{n}$ and $x:\tA \in \Gamma$, $x:\tA_{1} \in \Delta_{1}, \ldots, x:\tA_{n} \in \Delta_{n}$, then $\rho(x) \in \tei{\tA} \cap \tei{\tA_{1}} \cap \ldots \cap \tei{\tA_{n}}$.\\
From this we deduce that
$\rho \modelsr \Gamma, \rho \modelsr \Delta_{1}, \ldots, \rho \modelsr \Delta_{n}$.
From $\rho \modelsr \Gamma$, using the IH we deduce that $\tei{M}_{\rho} \in \ti{\cap_{1}^{n} \tT_{i} \to \tS} \silv{= \ti{\cap_{1}^{n} \tT_{i}} \to \ti{\tS}}$.
From $\rho \modelsr \Delta_{1}, \ldots, \rho \modelsr \Delta_{n}$,
using the IH we deduce that $\tei{N}_{\rho} \in \ti{\tT_{1}}, \ldots, \tei{N}_{\rho} \in \ti{\tT_{n}}$.
This means that  $\tei{N}_{\rho} \in \cap_{i}^{n}\tei{\tT_{i}}_{\rho} = \tei{\cap_{i}^{n}\tT_{i}}_{\rho}$. Using Lemma~\ref{lem:sub}(i) we obtain that $\tei{M}_{\rho} \tei{N}_{\rho}  = \tei{MN}_{\rho} \in \ti{\tS}$.
\item
The last rule applied is $(Weak)$, i.e.,\
$$\Gamma \vdashr M:\tA \Rightarrow \Gamma, x:\tB \vdashr \weak{x}{M}:\tA.$$
By the IH $\Gamma \modelsr M:\tA$.
Suppose that $\rho \modelsr \Gamma, x:\tB$ $\Leftrightarrow$  $\rho \modelsr \Gamma$ and $\rho \modelsr x:\tB$. From $\rho \modelsr \Gamma$ and $\Gamma \modelsr M:\tA$ we obtain $\tei{M}_{\rho} \in \ti{\tA}$.
But $\tei{\weak{x}{M}}_{\rho} = \tei{M}_{\rho}$ by Lemma~\ref{lem:sub}(iii), hence $\tei{\weak{x}{M}}_{\rho} \in \ti{\tA}$.
\item
The last rule applied is $(Cont)$, i.e.,\
$$\Gamma, x:\tA, y:\tB \vdashr M:\tS \Rightarrow \Gamma, z:\tA \cap \tB \vdashr \cont{z}{x}{y}{M}:\tS.$$
By the IH $\Gamma, x:\tA, y:\tB \modelsr M:\tS$.
Suppose that $\rho \modelsr \Gamma, z:\tA \cap \tB$, in order to prove $\tei{\cont{z}{x}{y}{M}}_{\rho} \in \ti{\tS}$. This means that $\rho \modelsr \Gamma$ and $\rho \modelsr z:\tA \cap \tB$ $\Leftrightarrow$ $\rho(z) \in \ti{\tA} \mbox{ and } \rho(z) \in \ti{\tB}$.
For the sake of simplicity let $\rho(z) \equiv N \in \LRc$. We define a new valuation $\rho'$ such that $\rho' = \rho (\isubs{N}{x}, \isubs{N}{y})$.
Then $\rho' \modelsr \Gamma, x:\tA, y:\tB$ since $x,y \not \in Dom(\Gamma)$, $N \in \ti{\tA}$ and $N \in \ti{\tB}$.
By the IH $\tei{M}_{\rho'} \in \ti{\tS}$.
By the definition of term valuation (Definition~\ref{def:val}) and Lemma~\ref{lem:sub}(iv)
we obtain
$\tei{M}_{\rho'} = \tei{M}_{\rho(\isubs{N}{x}, \isubs{N}{y})} = \tei{\cont{z}{x}{y}{M}}_{\rho}$,
since $\rho(z)=N$.
Hence, $\tei{\cont{z}{x}{y}{M}}_{\rho} \in \ti{\tS}$.

\end{itemize}
\end{proof}

\begin{lem}
\label{lem:SN}
Let $\tei{-}_\rho : \LR \to \LRc$ be a term valuation.
Then
$$\tei{M}_{\rho} \in \SN \;\;\Rightarrow \;\; M \in \SN.$$
\end{lem}

\begin{proof}
\silv{The proof is straightforward since although the terms $M$ and $\tei{M}_{\rho}$ are different, the latter is an instantiation of the former $\tei{M}_{\rho} = M[\isubs{N}{x}, \isubs{N}{y}, \isubs{\rho(x_{1})}{x_{1}}, \ldots, \isubs{\rho(x_{n})}{x_{n}}]$. This is a well-known property in $\lambda$-calculus and adding contraction and weakening still preserves this property. For this reason, the strong normalisation property is preserved.}
\end{proof}
\begin{thm} [SN for $\lR \cap$] \label{th:typ=>SN}
If $\Gamma \vdashr M:\tA$, then $M$ is strongly normalizing, i.e. $M \in \SN$.
\end{thm}

\begin{proof}
Suppose $\Gamma \vdashr M:\tA$.  
By Proposition~\ref{prop:sound}\; $\Gamma \modelsr M:\tA$. According to Definition~\ref{def:model}(iii), this means that for all $\rho$ such that $\rho \modelsr \Gamma$ we have that $\rho \modelsr M : \tA$.
\silv{Suppose $\Gamma = \{x_1:\alpha_1, \ldots x_k:\alpha_1k \}$. Since for each $\alpha_i$ by Proposition~\ref{prop:saturSets}(iv), $\ti{\tA_i}$ is saturated, this implies that the property $\INH(\ti{\alpha_i})$. Thus for each $\ti{\alpha_i}$
\[
(\exists n_i \geq 0) \; (\forall p \geq n_i) \;
                        \lambda x_1 \ldots \lambda x_p.\lambda y.y \in \ti{\alpha_i}.
\]
Let us obtain a particular valuation $\rho_0$ so that for each $x_i \in Dom(\Gamma)$ 
$$\rho_0(x_i) = \lambda x_1 \ldots \lambda x_{n_i}.\lambda y.y$$
Since $\INH(\ti{\alpha_i})$ it follows that  $\rho_{0}(x_i) \in \ti{\tA_i}$ for all $x_i:\tA_i \in \Gamma$.} Therefore, by Definition~\ref{def:model}(ii) $\rho_{0} \modelsr \Gamma$ and by Definition~\ref{def:model}(iii) $\rho_{0} \modelsr M: \tA$. Now by Definition~\ref{def:model}(i) we can conclude that $\tei{M}_{\rho_{0}} \in \ti{\tA}$. By Proposition~\ref{prop:saturSets}(iv) $\ti{\tA} \subseteq \SNc$, so by applying Lemma~\ref{lem:SN} we get $M \in \SNc$.

\end{proof}

\sil{There are a few differences with respect to the traditional reducibility method presented in~\cite{bare92}.
First of all, the interpretation of types only contains closed $\lR$-terms, as opposed to all $\lambda$-terms. Next, instead of $\VAR$ condition which ensures that the saturated sets contain variables, we introduce the condition $\INH$ which ensures that the terms of the form $\lambda_{1} \ldots \lambda_{p}.\lambda y.y$ belong to all $\mathcal{R}$-saturated sets. Further, valuations map variables to closed terms and term valuations map $\lR$-terms to closed $\lR$-terms.. Finally, in the proof of Theorem~\ref{th:typ=>SN}, instead of the valuation $\rho_0(x) = x$ which maps all the variables to themselves, we need a valuation $\rho_{0}(x) = \lambda x_{1} \ldots \lambda x_{n}.\lambda y.y.$ that maps every variable to a closed term (which is provided by $\INH(\ti{\beta})$).}



\subsection{Typeability $\Rightarrow$ SN in $\lGR\cap$}
\label{sec:typeSN_Gtz}

In this subsection, we prove the strong normalisation property of
the $\lGR$-calculi with intersection types. The termination is
proved by showing that the reductions on the sets $\LGR$ of the
typeable $\lGR$-expressions are included in particular
well-founded relations, which we define as the lexicographic
products of several well-founded component relations. The first
ones are based on the mappings of $\lGR$-expressions into
$\lR$-terms.
We show that these mappings preserve types
and that all $\lGR$-reductions can be simulated by the reductions
or identities of the corresponding $\lR$-calculi.
The other well-founded orders are based on the introduction of
quantities
designed to decrease a global measure associated with specific
$\lGR$-expressions during the computation.

\begin{defi}
\label{def:mapping} The mappings $\intt{\;\;}:\TGR\;\to\;\LR$ are
defined together with the auxiliary mappings $\intc{\;\;}:\KGR \;
\to \;(\LR \; \to \; \LR)$ in the following way:
$$
\begin{array}{lclclcl}
\intt{x} & = & x & & \intc{\bindx{t}}(M) & = & (\lambda x. \intt{t})M\\
\intt{\lambda x.t} & = & \lambda x. \intt{t} & & \intc{t::k}(M) & = & \intc{k}(M\intt{t})\\
\intt{\weak{x}{t}} & = & \weak{x}{\intt{t}} & & \intc{\weak{x}{k}}(M) & = & \weak{\jel{\{x\}\setminus Fv(M)}}{\intc{k}(M)}\\
\intt{\cont {x}{y}{z}{t}} & = & \cont{x}{y}{z}{\intt{t}} & & \intc{\cont {x}{y}{z}{k}}(M) & = &\cont{x}{y}{z}{\intc{k}(M)} \\
\intt{tk} & = & \intc{k}(\intt{t})\\
\end{array}
$$
\end{defi}

\begin{lem}
\rule{0in}{0in}
\begin{itemize}
\item[(i)]\label{lemma:pres-of-FV}
 $Fv(t) = Fv(\intt{t})$, for $t \in \TGR$.
\item[(ii)] \label{lemma:int-of-sub} $\intt{v\isub{t}{x}} =
\intt{v}\isub{\intt{t}}{x}$, for $v,t \in \TGR$.
\end{itemize}
\end{lem}
\begin{proof} The proof directly yields from 
Definition~\ref{def:mapping} and the substitution definitions.

\end{proof}

We prove that the mappings $\intt{\;\;}$ and $\intc{\;\;}$
preserve types. In the sequel, we use $\vdash_{\lR}$ to denote
derivations in $\lR \cap$. The notation $\LRvdash{`G}{\tA}$ stands
for $\{M \; \mid \; M \in \LR\;\; \&\;\; `G \vdash_{\lR} M:\tA\}$.

\begin{prop}[Type preservation with $\intt{\;\;}$]\label{prop:soundness}
\rule{0in}{0in}
\begin{itemize}
\item[(i)] If $\;`G \vdashr t:\tS$, then $`G \vdash_{\lR}
\intt{t}:\tS$.
\item[(ii)] If $\;`G;\jel{\cap^n_j}\tT_j \vdashr k:\tS$,
then $\intc{k}:\LRvdash{\jel{`G'_j}}{\tT_j} \to \LRvdash{`G \unc `G'
}{\tS}$, \jel{for all} $j \sil{\in \{1, \ldots, n\}}$ and for some $\jel{`G'= `G'_1 \sqcup ... \sqcup `G'_n}$.
\end{itemize}
\end{prop}
\begin{proof} The proposition is proved by simultaneous induction on
derivations. We distinguish cases according to the last typing
rule used.
\begin{itemize}
\item Cases $(Ax_{iw})$, $(Ax_{ew})$, $(\to_R)$, $(Weak_t)$ and
$(Cont_t)$ are easy, because the system $\lR\cap$ has exactly the
same rules. \item Case $(Sel)$: the derivation ends with the rule
$$
\infer[(Sel)]{`G'; \tA \vdashr \bindx t:\tS}
                    {`G', x:\tA \vdashr t:\tS}
$$
By IH we have that $`G', x:\tA \vdash_{\lR} \intt{t}:\tS$. For any
$M \in \LR$ such that $`G'' \vdash_{\lR} M:\tA$, for some $`G''$,
we have
\begin{center}
\[ \prooftree
    \prooftree
        `G', x:\tA \vdash_{\lR} \intt{t}:\tS
        \justifies
        `G' \vdash_{\lR} \lambda x.\intt{t}:\tA \to \tS
    \using{(\to_I)}
    \endprooftree
    `G'' \vdash_{\lR} M:\tA
   \justifies
   `G' \unc `G'' \vdash_{\lR}(\lambda x.\intt{t})M:\tS
   \using{(\to_E)}
\endprooftree
\]
\end{center}
Since $(\lambda x.\intt{t})M=\intc{\bindx t}(M)$, we conclude that
$\intc{\bindx t}:\LRvdash{`G''}{\tA} \to \LRvdash{`G' \unc
`G''}{\tS}$.
\item Case $(\to_L)$: the derivation ends with the
rule
$$
\infer[(\to_L)]{ `G \unc \Delta; \cap^m_j (\cap^n_i\tS_i \to \tT_j)
\vdashr t::k:\tR} {`G_1 \vdashr t:\tS_1\;...\;`G_n \vdashr t:\tS_n & \Delta;\cap^m_j\tT_j \vdash
k:\tR}
$$
\jel{where $`G = `G_1 \sqcup ... \sqcup `G_n$}. By IH we have that $\jel{`G_i} \vdash_{\lR} \intt{t}:\tS_i$, \jel{for
$i=1...n$}. For any $M \in \LR$ such that $\jel{`G'_j} \vdash_{\lR} M:\sil{\cap_{i}^{n}}\tS_i
\to \tT_j$, $\sil{j = 1, \ldots, m}$ we have
$$
\infer[(\to_E)]{`G \unc \jel{`G'_j} \vdash_{\lR} M \intt{t}:\tT_j}
                    {\jel{`G'_j} \vdash_{\lR} M:\sil{\cap_i^{n}}\tS_i \to \tT_j \quad `G \vdash_{\lR} \intt{t}:\tS_i}
$$
From the right-hand side premise in the $(\to_L)$ rule, by IH, we
get that $\intc{k}$ is the function with the scope
$\intc{k}:\LRvdash{\jel{`G'''_j}}{\tT_j} \to \LRvdash{`G''' \unc
`G''}{\tR}$, \jel{for $`G'''=`G'''_1 \sqcup ... \sqcup `G'''_n$}. For $`G'''\equiv `G \unc `G'$ and by taking
$M\intt{t}$ as the argument of the function $\intc{k}$, we get $`G
\unc \Delta \unc `G'\vdash_{\lR} \intc{k}(M\intt{t}):\tR$. Since
$\intc{k}(M\intt{t})=\intc{t::k}(M)$, we have that $`G \unc \Delta
\unc `G'\vdash_{\lR} \intc{t::k}(M):\tR$. This holds for any
$M$ of the appropriate type, yielding\\
$\intc{t::k}:\LRvdash{`G'}{\jel{\cap^n_i\tS_i} \to \tT_j} \to \LRvdash{`G
\unc \Delta \unc `G'}{\tR}$, which is exactly what we need. \item
Case $(Cut)$: the derivation ends with the rule
$$
\infer[(Cut)]{\sil{(`G \sqcup \ldots \sqcup `G_{n}) \unc \Delta} \vdashr tk:\tS}
                    {\sil{`G_{1} \vdashr t:\tT_1 \ldots `G_{n} \vdashr t:\tT_n} & \Delta; \sil{\cap \tT_i^{n}} \vdashr k:\tS}
$$
By IH we have that $`G \vdash_{\lR} \intt{t}:\tT_i$ and
$\intc{k}:\LRvdash{`G'}{\tT_i} \to \LRvdash{`G', \Delta}{\tS}$
\jel{for all} \sil{$i =1, \ldots, n.$} Hence, for any $M \in \Lambda^{\lR}$ such that
$`G' \vdash_{\lR} M:\tT_i$, $`G' \unc \Delta \vdash_{\lR}
\intc{k}(M):\tS$ holds. By taking $M \equiv \intt{t}$ and $`G' \equiv
`G$, we get $`G, \Delta \vdash_{\lR} \intc{k}(\intt{t}):\tS$. But
$\intc{k}(\intt{t})=\intt{tk}$, so the proof is done. \item Case
$(Weak_k)$: the derivation ends with the rule
$$
\infer[(Weak_k)]{`G, x:\tA; \tB \vdashr \weak{x}{k}:\tS}
                    {`G; \tB \vdashr k:\tS}
$$
By IH we have that $\intc{k}$ is the function with the scope
$\intc{k}:\LRvdash{`G'}{\tB} \to \LRvdash{`G \unc`G'}{ \tS}$,
meaning that for each $M \in \Lambda^{\lR}$ such that $`G'
\vdash_{\lR} M:\tB$ holds $`G'\unc `G'' \vdash_{\lR}
\intc{k}(M):\tS$. Now, we can apply $(Weak)$ rule:
$$
\infer[(Weak)]{`G \unc`G', x:\tA \vdashr
\weak{x}{\intc{k}(M)}:\tS}
                    {`G \unc `G' \vdashr \intc{k}(M):\tS}
$$
Since $\weak{x}{\intc{k}(M)}= \intc{\weak{x}{k}}(M)$, this means
that $\intc{\weak{x}{k}}:\LRvdash{`G'}{\tB} \to \LRvdash{`G \unc
`G',x:\tA}{\tS}$, which is exactly what we wanted to get. \item
Case $(Cont_k)$: similar to the case $(Weak_k)$, relying on the
rule $(Cont)$ in $\lR$. 
\end{itemize}
\end{proof}

For the given encoding $\intt{\;\;}$, we show that each
$\lGR$-reduction step can be simulated by $\lR$-reduction or
identity. In order to do so, we first prove the following lemmas
by induction on the structure of the contexts. The proofs of
Lemma~\ref{lemma:pi2-for-contexts} and
Lemma~\ref{lemma:int-of-append} also use Regnier's $\sigma$
reductions, investigated in \cite{regn94}.

\begin{lem}\label{lemma:context-closure}
If $M \to_{\lR} M'$, then $\intc{k}(M) \to_{\lR} \intc{k}(M').$
\end{lem}

\begin{lem}\label{lemma:pi2-for-contexts}
$\intc{k}((\lambda x.P)N) \to_{\lR} (\lambda x.\intc{k}(P))N.$
\end{lem}

\begin{lem}\label{lemma:int-of-append}
If $M \in \LR$ and $k,k' \in\KGR$, then $\intc{k'}\circ\intc{k}(M)
\to_{\lR} \intc{\app{k}{k'}}(M).$
\end{lem}

\begin{lem}
\rule{0in}{0in}
\begin{itemize}
\item[(i)] \label{lemma:prop-of-isub} If $x \notin Fv(k)$,  then
$(\intc{k}(M))\isub{N}{x} = \intc{k}(M\isub{N}{x}).$
\item[(ii)]\label{lemma:prop-of-cont} If $x,y \notin Fv(k)$,  then
$\cont{z}{x}{y}{(\intc{k}(M))} \to_{\lR}
\intc{k}(\cont{z}{x}{y}{M}).$
\item[(iii)]\label{lemma:push-of-weak} $\intc{k}(\weak{x}{M})
\to_{\lR} \weak{\jel{\{x\}\setminus Fv(k)}}{\intc{k}(M)}.$
\end{itemize}
\end{lem}

Now we can prove that the reduction rules of $\lGR$-calculi can
be simulated by the reduction rules or identities in the
corresponding $\lR$-calculi. \jel{Moreover, the equivalences of
$\lGR$-calculi are preserved in $\lR$-calculi.}

\begin{thm}[Simulation of $\lGR$-reductions by
$\lR$-reductions]\label{prop:simulation}
\rule{0in}{0in}
\begin{itemize}
\item[(i)] If a term $M\rightarrow M'$, then $\intt{M} \to_{\lR}
\intt{M'}$. \item[(ii)] If a context $k \rightarrow k'$ by
$\gamma_6$ or $\omega_6$ reduction, then $\intc{k}(M) \equiv
\intc{k'}(M)$, for any $M \in \LR$. \item[(iii)] If a context $k
\rightarrow k'$ by some other reduction, then $\intc{k}(M)
\to_{\lR} \intc{k'}(M)$, for any $M \in \LR$.
\item[(iv)] \jel{If $M \equiv M'$, then $\intt{M} \equiv_{\lR}
\intt{M'}$, and if $k \equiv k'$, then $\intc{k}(M)
\equiv_{\lR} \intc{k'}(M)$, for any $M \in \LR$.}
\end{itemize}
\end{thm}
\begin{proof} Without losing generality, we prove the statement only for
the outermost reductions. The rules $\jel{\gamma_0,} \sil{\gamma'_0,} \gamma_1, \gamma_2,
\gamma_3$, $\omega_1, \omega_2,
\omega_3$, $\gamma\omega_1$, $\gamma\omega_2$ are simulated by the
corresponding rules of $\lR$. For the remaining rules we have:
\begin{itemize}
\item[($\beta$)] $(\lambda x.t)(u::k) \to u(\bindx tk)$.

On the one hand $\intt{M} = \intt{(\lambda x.t)(u::k)} =
\intc{u::k}(\intt{\lambda x.t}) = \intc{k}((\lambda
x.\intt{t})\intt{u})$

On the other hand, $\intt{M'} = \intt{u(\bindx tk)} = \intc{\bindx
tk}(\intt{u}) = (\lambda x.\intt{tk})\intt{u} =\\
(\lambda x.\intc{k}(\intt{t}))\intt{u}$. So,
$\intt{M}\rightarrow_{\lR} \intt{M'}$ by
Lemma~\ref{lemma:pi2-for-contexts}.

\item[($\sigma$)] $T(\bindx{v})\to v\isub{T}{x}$.

$\intt{M}=\intt{T(\bindx{x})}=\intc{\bindx{x}}(\intt{T})=(\lambda
x.\intt{v})(\intt{T}),$

$\intt{M'}=\intt{v\isub{T}{x}}= \intt{v}\isub{\intt{T}}{x}$ by
Lemma~\ref{lemma:int-of-sub}, so $\intt{M} \to_{\lR} \intt{M'}$ by
$\beta$-reduction in the $\lR$-calculus.

\item[($\pi$)] $(tk)k' \to t(\append{k}{k'})$

$\intt{M} = \intt{(tk)k'} = \intc{k'}(\intt{tk}) =
\intc{k'}(\intc{k} (\intt{t}))$

$\intt{M'} = \intt{t(\append{k}{k'})} =
\intc{\append{k}{k'}}(\intt{t})$.

Applying Lemma~\ref{lemma:int-of-append} we get that $\intt{M}
\to_{\lR} \intt{M'}$.

\item[($\mu$)] $\bindx xk \to k$.

This reduction reduces context to context, so we have:

$\intc{\bindx xk}(M) = (\lambda x.\intt{xk})(M) = (\lambda
x.\intc{k}(x))(M)$.

This reduces to $\intc{k}(M)$ by $\beta$-reduction in the
$\lR$-calculus.

\item[($\gamma_4$)] $\cont{x}{x_1}{x_2}{(\bindy t)} \to \bindy
(\cont{x}{x_1}{x_2}{t})$

$\intc{K}(M) = \cont{x}{x_1}{x_2}{(\lambda y.\intt{t})M}.$

On the other hand,\\
$\intc{K'}(M) = (\lambda y.\cont{x}{x_1}{x_2}{\intt{t}})M.$

So $\intt{M} \to_{\lR} \intt{M'}$ by the rule $\gamma_2$.

\item[($\gamma_5$)] $\cont{x}{x_1}{x_2}{(t::k)} \to
(\cont{x}{x_1}{x_2}{t})::k, \;\;\;\mbox{if} \; x_1,x_2 \notin
Fv(k)$

$\intc{K}(M) = \cont{x}{x_1}{x_2}{(\intc{k}(M\intt{t}))}.$

$\intc{K'}(M) = \intc{k}(M(\cont{x}{x_1}{x_2}{\intt{t}}))$.

$x_1,x_2 \notin Fv(k)$ implies $x_1,x_2 \notin Fv(\intc{k}(M))$ so
we can apply Lemma~\ref{lemma:prop-of-cont} followed by reduction
$\gamma_3$ and conclude that $\intc{K}(M) \to_{\lR} \intc{K'}(M).$

\item[($\gamma_6$)] $\cont{x}{x_1}{x_2}{(t::k)} \to
t::(\cont{x}{x_1}{x_2}{k}), \;\;\;\mbox{if} \; x_1,x_2 \notin
Fv(t)$

$\intc{K}(M) = \intc{\cont{x}{x_1}{x_2}{(t::k)}}(M) =
\cont{x}{x_1}{x_2}{\intc{k}(M\intt{t})}.$

On the other hand,\\
$\intc{K'}(M) = \intc{t::(\cont{x}{x_1}{x_2}{k})}(M) =
\cont{x}{x_1}{x_2}{\intc{k}(M\intt{t})}.$

So $\intc{K}(M) = \intc{K'}(M)$.

\item[($\omega_4$)] $\bindx (\weak{y}{t}) \to \weak{y}{(\bindx
t)}, \;\;x \neq y$

$\intc{K}(M) = \intc{\bindx (\weak{y}{t})}(M) = (\lambda
x.\weak{y}{\intt{t}})M.$

$\intc{K'}(M) = \intc{\weak{y}{(\bindx t)}}(M) = \weak{y}{(\lambda
x.\intt{t})M}$.

So $\intc{K}(M) \to_{\lR} \intc{K'}(M)$ by the rule $\omega_2$.

\item[($\omega_5$)] $(\weak{x}{t})::k \to \weak{\jel{\{x\}\setminus Fv(k)}}{(t::k)}$

$\intc{K}(M) = \intc{(\weak{x}{t})::k}(M) =
\intc{k}(M\intt{\weak{x}{t}}) = \intc{k}(M\weak{x}{\intt{t}}).$

$\intc{K'}(M) = \intc{\weak{\jel{\{x\}\setminus Fv(k)}}{(t::k)}}(M) =
\weak{\jel{(\{x\}\setminus Fv(k))\setminus Fv(M)}}{\intc{t::k}(M)} = \weak{\jel{(\{x\}\setminus Fv(k))\setminus Fv(M)}}{\intc{k}(M\intt{t})}.$

Applying the rule $\omega_3$ of $\lR$ and
Lemma~\ref{lemma:push-of-weak} we get that

$\intc{K}(M) \rightarrow_{\lR} \intc{K'}(M)$.

\item[($\omega_6$)] $t::(\weak{x}{k}) \to \weak{\jel{\{x\}\setminus Fv(t)}}{(t::k)}$

$\intc{K}(M) = \intc{t::(\weak{x}{t})}(M) =
\intc{\weak{x}{k}}(M\intt{t}) = \weak{\jel{\{x\}\setminus Fv(M\intt{t})}}{\intc{k}(M\intt{t})} =
\intc{K'}(M).$

\jel{\noindent The proof of $(iv)$ is trivial, since the equivalences of $\lGR$-calculi and $\lR$-calculi coincide.}
\end{itemize}
\end{proof}

\jel{The previous proposition shows that $\beta$, $\pi$, $\sigma$, $\mu$, \jel{$\gamma_0$} -
$\gamma_5$, $\omega_1$ - $\omega_5$, $\gamma\omega_1$ and
$\gamma\omega_2$ $\lGR$-reductions are
interpreted by $\lR$-reductions and that $\gamma_6$ and $\omega_6$ $\lGR$-reductions are
interpreted by an identity in the $\lR$. Since the set of equivalences of the two bases of the resource control cube coincide, they are trivially preserved.} If one
wants to prove that there is no infinite sequence of
$\lGR$-reductions one has to prove that there cannot exist an
infinite sequence of $\lGR$-reductions which are all interpreted
as identities. To prove this, one shows that if a term is reduced
with such a $\lGR$-reduction, it is reduced for another order that
forbids infinite decreasing chains. This order is itself composed
of several orders, free of infinite decreasing chains
(Definition~\ref{def:ord}).

\begin{defi}
\label{def:x} \label{def:cnorm}\label{def:wnorm}
The functions
$\mathcal{S}^{\mathcal{R}},\;\cnorm{\;},\;\wnorm{\;}: \LGR \to 
\mathbb{N}$ are defined in Figure~\ref{fig:size-norms}.
\end{defi}
\vspace*{-0.3em}
\begin{figure}[hbtp]
\centerline{ \framebox{ $
\begin{array}{rcl|rcl|rcl}\\
\size{x} & = & 1 &  \cnorm{x} & = & 0 & \wnorm{x} & = & 1\\
\size{\lambda x.t} & = & 1 + \size{t} & \cnorm{\lambda x.t} & = & \cnorm{t} & \wnorm{\lambda x.t} & = & 1 + \wnorm{t}\\
\size{\weak{x}{e}} & = & 1 + \size{e} & \cnorm{\weak{x}{e}} & = & \cnorm{e} & \wnorm{\weak{x}{e}} & = & 0\\
\size{\cont {x}{y}{z}{e}} & = & 1 + \size{e} &
\cnorm{\cont{x}{y}{z}{e}} & = & \cnorm{e}  + \size{e} &
\wnorm{\cont {x}{y}{z}{e}} & = & 1 + \wnorm{e}\\
\size{tk} & = & \size{t} + \size{k} & \cnorm{tk} & = & \cnorm{t} + \cnorm{k} & \wnorm{tk} & = & 1 + \wnorm{t} + \wnorm{k}\\
\size{\bindx{t}} & = & 1 + \size{t} & \cnorm{\bindx{t}} & = & \cnorm{t} & \wnorm{\bindx{t}} & = & 1 + \wnorm{t}\\
\size{t::k} & = & \size{t} + \size{k} & \cnorm{t::k} & = & \cnorm{t} + \cnorm{k} & \wnorm{t::k} & = & 1 + \wnorm{t} + \wnorm{k}\\ \\
\end{array}
$ }} \caption{Definitions of $\size{e}, \cnorm{e}, \wnorm{e}$}
\label{fig:size-norms}
\end{figure}

\begin{lem}\label{lemma:cnorm}
For all $e, e' \in \lGR$:
\begin{itemize}
\item[(i)] If $e \; \to_{\gamma_6} \; e'$, then $\cnorm{e} >
\cnorm{e'}$. \item[(ii)] If $e \; \to_{\omega_6} \; e'$, then
$\cnorm{e} = \cnorm{e'}$.
\jel{\item[(iii)] If $e \; \equiv \; e'$, then
$\cnorm{e} = \cnorm{e'}$.}
\end{itemize}
\end{lem}

\begin{lem}\label{lemma:wnorm}
For all $e, e' \in \lGR$:
\begin{itemize}
\item[(i)]  If $e \; \to_{\omega_6} \; e'$, then
$\wnorm{e} > \wnorm{e'}$.
\jel{\item[(ii)]  If $e \; \equiv \; e'$, then
$\wnorm{e} = \wnorm{e'}$.}
\end{itemize}
\end{lem}

Now we can define the following orders based on the previously
introduced mappings and norms.
\begin{defi}\label{def:ord}
We define the following strict orders and equivalences on $\LGR\cap$:
\begin{itemize}
\item[(i)] $t >_{\lR} t'$\; iff \; $\intt{t} \rightarrow^+_{\lR} \intt{t'}$;\;\;  $t =_{\lR} t'$\; iff \; $\intt{t} \equiv \intt{t'}$; \\
          $k >_{\lR} k'$\; iff \; $\intc{k}(M) \rightarrow^+_{\lR} \intt{k'}(M)$\; for every $M \in \LR$ ;\\
          $k =_{\lR} k'$\; iff \; $\intc{k}(M) \equiv \intc{k'}(M)$\; for every $M \in \LR$;
\item[(ii)]  $e >_\mathsf{c} e'$\; iff \; $\cnorm{e} >
\cnorm{e'}$;\;\; $e =_\mathsf{c} e'$\; iff \;$\cnorm{e} =
\cnorm{e'};$ \item[(iii)] $e >_\mathsf{w} e'$\; iff \;$\wnorm{e} >
\wnorm{e'}$;\;\; $e =_\mathsf{w} e'$\; iff \; $\wnorm{e} =
\wnorm{e'};$
\end{itemize}
\end{defi}

A lexicographic product of two orders $>_1$ and $>_2$ is
usually defined as follows (\cite{baadnipk98}):\\
$\hspace*{20mm} a >_1 \times_{lex} >_2 b \; \Leftrightarrow \;a
>_1 b \;\;or\;\; (a =_1 b \;and \; a >_2 b). $

\begin{defi}
We define the relations $\ggr$ on  $\LGR$ as the lexicographic
products:
\[ \ggr \;\;\;= \;\;\;>_{\lR}\; \times_{lex} \; >_\mathsf{c} \; \times_{lex} \;  >_\mathsf{w}.\]
\end{defi}

The following proposition proves that the reduction relation on
the set of typed $\lGR$-expressions is included in the given
lexicographic product $\ggr$.
\begin{prop}\label{prop:inclusion}
For each $e \in \LGR$: if  $e \to e'$, then $e \ggr e'$.
\end{prop}
\begin{proof}
The proof is by case analysis on the kind of reduction and the structure of $\ggr$.\\
If $e \to e'$ by $\beta$,  $\sigma$, $\pi$, $\mu$, \jel{$\gamma_0$}, \sil{$\gamma^{\;\;'}_0$,} $\gamma_1$
$\gamma_2$, $\gamma_3$, $\gamma_4$ $\gamma_5$, $\omega_1$,
$\omega_2$, $\omega_3$ $\omega_4$ or $\omega_5$, $\gamma\omega_1$, $\gamma\omega_2$, reduction, then $e
>_{\lR} e'$
by Proposition~\ref{prop:simulation}.\\
If $e \to e'$ by $\gamma_6$, then $e =_{\lR} e'$ by
Proposition~\ref{prop:simulation}, and $e >_\mathsf{c} e'$ by Lemma~\ref{lemma:cnorm}.\\
Finally, if $e \to e'$ by $\omega_6$, then $e =_{\lR} e'$ by
Proposition~\ref{prop:simulation}, $e =_\mathsf{c} e'$ by
Lemma~\ref{lemma:cnorm} and $e >_\mathsf{w} e'$ by
Lemma~\ref{lemma:wnorm}. 
\end{proof}

SN of $"->"$ is another terminology for the well-foundedness of the
relation $"->"$ and it is well-known that a relation included in a
well-founded relation is well-founded and that the lexicographic
product of well-founded relations is well-founded.

\begin{thm}[Strong normalization]\label{th:typ=>SN-Gtz}
Each expression in $\LGR \cap$ is SN.
\end{thm}
\begin{proof}
  The reduction $"->"$ is well-founded on $\LGR\cap$  as it is
  included (Proposition~\ref{prop:inclusion}) in the relation $\ggr$ which is well-founded as the lexicographic
  product of the well-founded relations $>_{\lR}$, $>_\mathsf{c}$ and
  $>_\mathsf{w}$. The relation $>_{\lR}$ is based on the interpretation
  $\intt{~}:\LGR "->" \LR$.  By Proposition~\ref{prop:soundness} typeability is
preserved by the interpretation $\intt{~}$ and $"->"_{\lR}$ is SN
(i.e., well-founded) on $\LR\cap$
(Section~\ref{sec:reducibility}), hence $>_{\lR}$ is well-founded
on $\LGR\cap$.
  Similarly, $>_\mathsf{c}$ and $>_\mathsf{w}$ are well-founded, as they are based on
  interpretations into the well-founded relation $>$ on the set~$\mathbb{N}$ of natural
  numbers.

\end{proof}


\section{SN $\Rightarrow$ Typeability in all systems of the resource control cube}
\label{sec:SNtypeBoth}



\silv{Let us turn our attention to the most unique property of intersection types systems that all strongly normalising terms are typeable by intersection types. We will prove this property first for $\lR$-terms and then for $\lGR$-terms.}

\subsection{SN $\Rightarrow$ Typeability in $\lR \cap$}
\label{sec:SNtype}

We want to prove that if a $\lR$-term is SN, then it is typeable
in the system $\lR \cap$. We proceed in two steps:
\begin{enumerate}
\item we show that all $\lR$-normal forms are typeable and
\item we prove the head subject expansion property.
\end{enumerate}
First, let us observe the structure of
the $\lR$-normal forms, given by the following abstract syntax:
$$
\begin{array}{rcl}
 M_{nf} & ::= & x\,|\,\lambda
x.M_{nf}\,|\,xM^1_{nf}...M^n_{nf}\,|\,\lambda x.
\weak{x}{M_{nf}}\SKIP{,\;\mbox{if}\;\;x \notin Fv(M_{nf})}\\
& & |\,\cont{x}{x_1}{x_2}M_{nf}N_{nf},\;\;\mbox{if}\;\;x_1
\in Fv(M_{nf}),\;x_2 \in Fv(N_{nf})\\
W_{nf} & ::= & \weak{x}{M_{nf}}\SKIP{,\;\mbox{if}\;\;x \notin Fv(M_{nf})}\,|\,\weak{x}{W_{nf}}\SKIP{,\;\mbox{if}\;\;x \notin Fv(W_{nf})}\\
\end{array}
$$

\jel{Notice that it is necessary to distinguish normal forms $W_{nf}$ since the term $\lambda x.
\weak{y}{M_{nf}}$ is not a normal form, i.e. $\lambda x.
\weak{y}{M_{nf}} \to_{\omega_1} \weak{y}{\lambda x.M_{nf}}$.}


\begin{prop}\label{prop:nf-are-typed}
$\lR$-normal forms are typeable in the system $\lR \cap$.
\end{prop}
\begin{proof} By an easy induction on the structure of $M_{nf}$ and
$W_{nf}$. Notice that the typing rules for introducing the
explicit resource control operators change only the left-hand side
of the sequents while the types of the expressions on the
right-hand side stay unchanged.
\end{proof}

\begin{prop}[Inverse substitution lemma]\label{prop:inv-subst-lemma}
Let $\;\Gamma \vdashr M\isub{N}{x}:\tS\;$ and $N$ typeable. Then,
there are $\Gamma'$, \jel{$\Delta = \Delta_1 \sqcup...\sqcup
\Delta_n$ and $\tT_i, \;i=1,...,n$} such that $\Gamma = \Gamma'
\unc \Delta$, $\;\jel{\Delta_i}\vdashr N:\tT_i\;$ \jel{for all}
$\sil{i=1, \ldots, n}$ and $\Gamma', x:\jel{\cap^n_i}\tT_i \vdashr M:\tS$.
\end{prop}
\begin{proof} By induction on the structure of $M$. We will just show the
basic case and the cases related to resource operators.
\begin{itemize}
\item Basic case:\\
- $\;M\equiv x$. Then $M\isub{N}{x}=x\isub{N}{x}=N$. For $\Gamma'
\equiv \emptyset$, $\Delta \equiv \Gamma$ and $\tT_i \equiv \tS$
we get $\;\Delta \vdashr N:\tS$ from the premise and
$x:\jel{\cap^n_i\tS_i}\vdashr x:\tS$ from the axiom.\\
- $\;M\equiv y$. Then $M\isub{N}{x}=y\isub{N}{x}=y$. This case is
possible only if $\w \notin \mathcal{R}$, so we can assume that
$\Gamma = \Gamma \sqcup \Delta$, where $\Delta \vdashr N:\tT$,
from the premise that $N$ is typeable, and $\Gamma \vdashr y:\tS$
from the implicit weakening axiom.
\item Case $\;M\equiv
\weak{x}{M'}$. Then $M\isub{N}{x}=(\weak{x}{M'})\isub{N}{x}=
\weak{Fv(N)\jel{\setminus Fv(M')}}{M'}$. From the premise $\Gamma \vdashr
(\weak{x}{M'})\isub{N}{x}: \tS$ we have $\Gamma \vdashr
\weak{Fv(N)\jel{\setminus Fv(M')}}{M'}: \tS$, hence by the generation lemma $\Gamma'
\vdashr {M'}: \tS$, for $\Gamma' = \Gamma \setminus (Fv(N)\jel{\setminus Fv(M')})$. Now,
for an arbitrary type $\tA$, we have  $\Gamma', x:\tA \vdashr
\weak{x}{M'}: \tS$. From the premise that $N$ is typeable, knowing
that $\w \in \mathcal{R}$, we get $\Delta \vdashr N:\tB$ where
$Dom(\Delta) = Fv(N)$. The proposition is proved by taking $\tA
\equiv \tB$.
\item Case $\;M\equiv \weak{y}{M'}$. Then
$M\isub{N}{x}=(\weak{y}{M'})\isub{N}{x}= \weak{\jel{\{y\}\setminus Fv(N)}}{M'}\isub{N}{x}$.
From $\Gamma \vdashr \weak{\jel{\{y\}\setminus Fv(N)}}{M'}\isub{N}{x}: \tS$ by generation
lemma, we have that $\Gamma' = \Gamma \jel{\setminus (\{y\}\setminus Fv(N))}, y:\tA$ and $\Gamma' \vdashr
M'\isub{N}{x}: \tS$. Now, by IH we get that $\Gamma' = \Gamma''
\unc \Delta$, \jel{$\Delta = \Delta_1 \sqcup ... \sqcup \Delta_n$}, $\;\jel{\Delta_i} \vdashr N:\tT_i$ \jel{for all} $\sil{i=1, \ldots, n}$ and
$\Gamma'', x:\jel{\cap^n_i}\tT_i \vdashr M':\tS$. Since $y
\notin M'$, we get $\Gamma'', y:\tA, x:\jel{\cap^n_i}\tT_i \vdashr
\weak{y}{M'}:\tS$. \item Case $\;M\equiv \cont{y}{y_1}{y_2}{M'}$
and $x \neq y$. Then
$M\isub{N}{x}=(\cont{y}{y_1}{y_2}{M'})\isub{N}{x}=
\cont{y}{y_1}{y_2}{M'}\isub{N}{x}$.  From the premise $\Gamma
\vdashr \cont{y}{y_1}{y_2}{M'\isub{N}{x}}:\tS$ using the
generation lemma we get that $\Gamma = \Gamma', y:\tA \cap \tB$
and $\Gamma', y_1:\tA,y_2:\tB \vdashr M'\isub{N}{x}:\tS$.  By IH
we get that $\Gamma' = \Gamma'' \unc \Delta$, \jel{$\Delta = \Delta_1 \sqcup ... \sqcup \Delta_n$}, $\;\jel{\Delta_i} \vdashr
N:\tT_i$ \jel{for all} $\sil{i=1, \ldots, n}$ and $\Gamma'', y_1:\tA,y_2:\tB,
x:\jel{\cap^n_i}\tT_i \vdashr M':\tS$. Using $(Cont)$ rule we
get $\Gamma'', y:\tA \cap \tB, x:\jel{\cap^n_i}\tT_i \vdashr
\cont{y}{y_1}{y_2}{M'}:\tS$.
\item Case $\;M\equiv \cont{x}{x_1}{x_2}{M'}$. Then
$M\isub{N}{x}=(\cont{x}{x_1}{x_2}{M'})\isub{N}{x}=
\cont{Fv(N)}{Fv(N_1)}{Fv(N_2)}{\jel{M'[N_1/x_1,N_2/x_2]}}$.
From the premise $\Gamma \vdashr
\cont{Fv(N)}{Fv(N_1)}{Fv(N_2)}{\jel{M'[N_1/x_1,N_2/x_2]}}:\tS$,
using the generation lemma we get that for $Fv(N) =
\{y_1,...,y_n\}$ holds $\Gamma = \Gamma', y_1:\tA_1 \cap
\tB_1,...,y_n:\tA_n \cap \tB_n $ and $\Gamma',
y'_1:\tA_1,y''_1:\tB_1,...,y'_n:\tA_n,y''_n:\tB_n  \vdashr
\jel{M[N_1/x_1,N_2/x_2]}:\tS$. Applying IH two times, we
obtain $\Gamma' = \Gamma'' \unc \Delta' \unc \Delta''$, \jel{where $\Delta' = \Delta'_1 \sqcup ... \sqcup \Delta'_n$ and
$\Delta'' = \Delta''_1 \sqcup ... \sqcup \Delta''_n$}, $\;\Delta'_i
\vdashr N_1:\tT'_i$ \jel{for all} $\sil{i=1, \ldots, n}$,  $\;\Delta''_i \vdashr
N_2:\tT''_i$ \jel{for all} $\sil{i=1, \ldots, n}$, and $\Gamma'',
y'_1:\tA_1,y''_1:\tB_1,...,y'_n:\tA_n,y''_n:\tB_n,
x_1:\jel{\cap^n_i}\tT'_i, x_2:\jel{\cap^n_i}\tT''_i \vdashr
M':\tS$. Since $N_1$ and $N_2$ are obtained by renaming $N$ we
have that $\jel{\cap^n_i}\tT'_i \equiv \jel{\cap^n_i}\tT''_i \equiv \jel{\cap^n_i}\tT_i$,
$\;\Delta_i \vdashr N:\tT_i$ \jel{for all} $\sil{i=1, \ldots, n}$ and for $\jel{\Delta_i =
\Delta'_i \unc \Delta''_i}$. Finally, by $(Cont)$ rule we get $\Gamma'',
y'_1:\tA_1,y''_1:\tB_1,...,y'_n:\tA_n,y''_n:\tB_n,
x:\jel{\cap^n_i}\tT_i \vdashr \cont{x}{x_1}{x_2}{M'}:\tS$ and the
proof is done. 
\end{itemize}
\end{proof}

\begin{prop}[Head subject expansion]
\label{prop:sub-exp} For every $\lR$-term $M$: if $M \to
M'$, $M$ is contracted redex and $\;\Gamma \vdashr M':\tS\;$, then
$\;\Gamma \vdashr M:\tS$, provided that if $M \equiv (\lambda
x.N)P \to_{\beta} N\isub{P}{x} \equiv M'$, $\;P$ is typeable.
\end{prop}
\begin{proof} By the case study according to the applied reduction. 
\end{proof}

\begin{thm}[SN $\Rightarrow$ typeability]\label{thm:SNtypable}
All strongly normalising $\lR$-terms are typeable in the $\lR\cap$
system.
\end{thm}
\begin{proof} By induction on the length of the longest
reduction path out of a strongly normalising term $M$, with a
subinduction on the size of $M$. We also use \jel{Proposition~\ref{prop:subequiv}}.
\begin{itemize}
\item If $M$ is a normal form, then $M$ is typeable by
Proposition~\ref{prop:nf-are-typed}. \item If  $M$ is itself a
redex, let $M'$ be the term obtained by contracting the redex $M$.
$M'$ is also strongly normalising, hence by IH it is typeable.
Then $M$ is typeable, by Proposition~\ref{prop:sub-exp}. Notice
that, if $M \equiv (\lambda x.N)P \to_{\beta} N\isub{P}{x} \equiv
M'$, then, by IH, $P$ is typeable - since the length of the
longest reduction path out of $P$ is not larger than that of $M$,
and the size of $P$ is smaller than the size of $M$. \item Next,
suppose that  $M$ is not itself a redex nor a normal form. Then
$M$ is of one of the following forms: $\lambda x.N$, $\lambda
x.\weak{x}{N}$, $xN_1...N_m$, $\weak{x}{N}$, or
$\cont{x}{x_1}{x_2}{NP},\;x_1 \in Fv(N),\;x_2 \in Fv(P)$ (in each
case some $N_i$ or $P$ are \emph{not} in normal form). $N_i$ and
$P$ are typeable by IH, as subterms of $M$. Then, it is easy to
build the typing for $M$. 
\end{itemize}
\end{proof}



\subsection{SN $\Rightarrow$ Typeability in $\lGR \cap$}
\label{sec:SNtype_Gtz}

Finally, we want to prove that if a $\lGR$-term is SN, then it is
typeable in the system $\lGR \cap$. We follow the procedure used
in Section~\ref{sec:SNtype}. The proofs are similar to the ones in
Section~\ref{sec:SNtype} and are omitted.

The abstract syntax of $\lGR$-normal forms is the following:

\hspace*{5mm}
$
\begin{array}{rcl}
 t_{nf} & ::= & x\,|\,\lambda
x.t_{nf}\,|\,x(t_{nf}::k_{nf})\,|\,\lambda x.
\weak{x}{t_{nf}}\SKIP{,\;\mbox{if}\;\;x \notin Fv(t_{nf})}
\,|\, \cont{x}{y}{z}{y(t_{nf}::k_{nf})}\\
k_{nf} & ::= & \bindx t_{nf}\,|\,t_{nf}::k_{nf}\,|\,\bindx
\weak{x}{t_{nf}}\SKIP{,\;\mbox{if}\;\;x \notin Fv(t_{nf})}
\,|\, \cont{x}{y}{z}{(t_{nf}::k_{nf})},\;y \in Fv(t_{nf}), z \in Fv(k_{nf})\\
w_{nf} & ::= & \weak{x}{e_{nf}}\SKIP{,\;\mbox{if}\;\;x \notin
Fv(e_{nf})}\,|\,\weak{x}{w_{nf}}\SKIP{,\;\mbox{if}\;\;x \notin Fv(w_{nf})}
\end{array}
$

We use $e_{nf}$ for any $\lGR$-expression in normal form.
\begin{prop}\label{prop:nf-are-typ}
$\lGR$-normal forms are typeable in the system $\lGR \cap$.
\end{prop}
\begin{proof} The proof goes by an easy induction on the structure of
$e_{nf}$. 
\end{proof}

\noindent The following two lemmas explain the behavior of the
meta operators $\isub{\;}{\;}$ and $\app\;$  during expansion.

\begin{lem}[Inverse substitution lemma]\label{prop:inv-subst-Gtz}
\rule{0in}{0in}
\begin{itemize}
\item[(i)] Let $\;\Gamma \vdashr t\isub{u}{x}:\tS\;$ and $u$
typeable. Then, there exist \jel{$\Delta = \Delta_1 \sqcup ... \sqcup \Delta_n$} and \jel{$\cap^n_i\tT_i, \;i=1,...,n$}
such that $\jel{\Delta_i} \vdashr u:\tT_i$ \jel{for all} $\sil{i=1, \ldots, n}$ and $\Gamma',
x:\cap\tT_i \vdashr t:\tS$, where $\Gamma = \Gamma' \unc \Delta$.
\item[(ii)] Let $\;\Gamma;\tA \vdashr k\isub{u}{x}:\tS\;$ and $u$
typeable. Then, there exist \jel{$\Delta = \Delta_1 \sqcup ... \sqcup \Delta_n$} and \jel{$\cap^n_i\tT_i, \;i=1,...,n$}
such that $\jel{\Delta_i}\vdashr u:\tT_i$ \jel{for all} $\sil{i=1, \ldots, n}$ and $\Gamma',
x:\cap\tT_i;\tA \vdashr k:\tS$, where $\Gamma = \Gamma' \unc
\Delta$.
\end{itemize}
\end{lem}
\begin{proof} The proof goes straightforward by mutual induction on the
structure of terms and contexts. 
\end{proof}

\begin{lem}[Inverse append lemma]
\label{prop:inverse-app} If $~~\Gamma;\tA\vdashr \app{k}{k'}:\tS$,
then $\Gamma=\Gamma' \unc \Gamma''$ \jel{where $\Gamma'= \Gamma'_1 \sqcup ... \sqcup \Gamma'_n$} and there is a type
\jel{$\cap^n_i\tT_i, \;i=1,...,n$} such that $\jel{\Gamma'_i};\tA\vdashr k:\tT_i\;$ \jel{for all}
$\sil{i=1, \ldots, n}$ and $\Gamma'';\jel{\cap^n_i\tT_i}\vdashr k':\tS$.
\end{lem}
\begin{proof} The proof goes by the induction on the structure of the
context $k$.

\end{proof}

Now we prove that the type of an expression is preserved during
the expansion.

\begin{prop}[Head subject expansion]
\label{prop:sub-exp-Gtz}
\rule{0in}{0in}
\begin{itemize}
\item[(i)] For every $\lGR$-term $t$: if $t \to t'$, $t$ is
contracted redex and $\;\Gamma \vdashr t':\tS\;$, then $\;\Gamma
\vdashr t:\tS$. \item[(ii)] For every $\lGR$-context $k$: if $k
\to k'$, $k$ is contracted redex and $\;\Gamma;\tA \vdashr
k':\tS\;$, then $\;\Gamma;\tA \vdashr k':\tS$.
\end{itemize}
\end{prop}
\begin{proof}
The proof goes by the case study according to the applied
reduction. 
\end{proof}

\begin{thm}[SN $\Rightarrow$ typeability]\label{thm:SNtypable-Gtz}
All strongly normalising $\lGR$ expressions are typeable in the
$\lGR\cap$ systems.
\end{thm}
\begin{proof} The proof is by induction over the length of the longest
reduction path out of a strongly normalising expression $e$, with
a subinduction on the size of $e$. We also use \jel{Proposition~\ref{prop:subequiv-sequent}}. 
\end{proof}

\silv{The  complete characterisation of strongly normalising terms by intersection types
\begin{itemize}
\item in the natural deduction ND-base of the resource control cube is a corollary to Theorem \ref{th:typ=>SN} and \ref{thm:SNtypable}
\item in the sequent LJ-base of the resource control cube is a corollary to Theorem \ref{th:typ=>SN-Gtz} and \ref{thm:SNtypable-Gtz}.
\end{itemize}

\begin{thm}[Complete characteristation of SN]
\rule{0in}{0in}
\begin{itemize}
\item[-] A $\lR$-term is strongly normalising if and only if it is typeable in $\lR \cap$.
\item[-] A $\lGR$-term is strongly normalising if and only if it is typeable in $\lGR \cap$.
\end{itemize}
\end{thm}
} 
\section{Conclusions}
\label{sec:conclusions}

In this paper, we  proposed intersection type assignment
systems for $\lR$ and $\lGR$-calculi,
\sil{two systems of lambda calculi parametrized with respect to $\mathcal{R} \subseteq \{\co, \w\}$, where}
$\co$ is a contraction and $\w$ is a weakening.
These two families of lambda calculi form the so-called \emph{resource control
cube}. Four $\lR$-calculi form the ``natural deduction base" of the
cube, corresponding to the ``implicit base'' of~\cite{kesnrena09},
whereas ``the sequent base'' contains four $\lGR$-calculi,
generalization of $\ell\lambda^{\mathsf{Gtz}}$-calculus of~\cite{ghilivetlesczuni11}.
In each base, the calculi differ by the implicit/explicit treatment
of the resource operators contraction and weakening.

The intersection type systems proposed here, for resource control
lambda and sequent lambda calculus, give a complete
characterisation of strongly normalising terms for all eight
calculi of the resource cube.
We propose general proofs for each base handled by
various side conditions in some cases.
In order to prove the strong normalisation of typeable
resource lambda terms, we use an appropriate modification of
the reducibility method. The same property
for resource sequent lambda expressions is proved by using a well-founded
lexicographic order based on suitable embedding into the former
calculi. 
This paper expands the range of
the intersection type techniques and combines different methods in
the strict types environment. Unlike the approach of introducing
non-idempotent intersection into the calculus with some kind of
resource management \cite{pagaronc10}, our intersection is
idempotent.

It would be interesting to investigate the relation between
the resource control enabled via explicit operators used
here and the approach used in \cite{pagaronc10}, where the
resources are managed via applicative bags with multiplicities.
Another direction will involve the investigation of the
use of intersection types in constructing models for sequent lambda calculi,
since intersection types are known powerful means for building
models of lambda calculus (\cite{barecoppdeza83,dezaghillika04}).
On the other
hand it should be noticed that the substitutions in our $\lGR$-calculi
are implicit.
Considering explicit susbtitutions would complete the sequent part
of the cube as Kesner and Renaud have done for $\lR$ in their
prismoid.

Furthermore, resource control lambda and sequent lambda calculi
are good candidates to investigate the computational content of
substructural logics~\cite{schrdose93}, both in natural deduction
and sequent calculus.  The motivation for these logics comes from
philosophy (Relevant and Affine  Logic), linguistics (Lambek Calculus) and
computing (Linear Logic). The basic idea of resource control is to
explicitly handle structural rules, so that the absence of (some)
structural rules in substructural logics such as weakening,
contraction, commutativity, associativity can possibly be handled
by resource control operators. This is in the domain of further
research. 

\pierre{From a more pragmatic perspective,  resources need to be controlled tightly in computer
applications.  For instance, Kristoffer Rose has undertaken the description of compilers by rules
with
binders~\cite{rose11:_implem_trick_that_make_crsx_tick,rose:LIPIcs:2011:3130}. He noticed that
the implementation of substitutions of linear variables by
inlining is efficient, whereas substitutions of duplicated
variables require a cumbersome and time consuming mechanism, based on
pointers. It  is therefore important to precisely control
duplications.  On the other hand, strong control of erasing does
not require a garbage collector and prevents memory leaking.}
\silv{Another line of application of resource control is related to object-oriented languages.
Alain Mycroft~\cite{mycr11} presented resource aware type-systems for multi-core program efficiency.
In this framework an identified ``memory isolation" property enables multi-core programs to avoid slowdown due to cache contention. The existing work on Kilim and its isolation-type system is related to both substructural types and memory isolation.
}

Finally, the two calculi with both resource control operators
explicit, namely $\lambda_{{\tt C W}}$ of \cite{kesnrena09} and
$\ell\lambda^{\mathsf{Gtz}}$-calculus of \cite{ghilivetlesczuni11} deserve particular
attention. Due to the multiplicative style of the typing rules,
and the reductions' orientation of propagating the contraction in
the term and extracting the weakening out of the term, these
calculi exhibit optimization in terms of the minimal total size of
the bases used for the type assignments. The consequences of this
property, particularly for the implementation related issues,
should be investigated.\\

\textbf{Acknowledgements} Above all, our gratitude goes to two anonymous referees for the present submission. Their extensive and detailed comments helped us tremendously to improve our work. We would also like to thank Dragi\v sa \v Zuni\' c for fruitful discussion.
\bibliographystyle{alpha}
\begin{normalsize}

\end{normalsize}

\end{document}